\documentclass[12pt]{amsart}

\usepackage[matrix,arrow,curve,frame]{xy}
\usepackage{amsmath,amsthm,amssymb,enumerate}
\usepackage{latexsym}
\usepackage{amscd}
\usepackage[colorlinks=false]{hyperref}
\usepackage{euscript}

\newtheorem{thm}[subsection]{Theorem}

\newtheorem{prob}[subsection]{Problem}

\newtheorem{cor}[subsection]{Corollary}
\newtheorem{lemma}[subsection]{Lemma}
\newtheorem{conj}[subsection]{Conjecture}

\theoremstyle{definition}

\numberwithin{equation}{section}

\def\C{{\bf C}}

\def\cO{{\cal O}}

\def\cP{{\cal P}}
\def\cA{{\cal A}}

\def\ra{\rightarrow}
\def\bra{\langle}
\def\ket{\rangle}

\def\C{{\bf C}}

\def\cA{{\mathcal A}}
\def\cB{{\mathcal B}}

\def\cD{{\mathcal D}}
\def\cE{{\mathcal E}}
\def\cF{{\mathcal F}}

\def\cI{{\mathcal I}}
\def\cJ{{\mathcal J}}

\def\cM{{\mathcal M}}

\def\cO{{\mathcal O}}
\def\cP{{\mathcal P}}

\def\cR{{\mathcal R}}
\def\cS{{\mathcal S}}

\def\cV{{\mathcal V}}
\def\cW{{\mathcal W}}

\def\gg{{\mathfrak g}}

\def\gl{{\mathfrak l}}

\def\go{{\mathfrak o}}
\def\gp{{\mathfrak p}}

\def\gs{{\mathfrak s}}

\newfont{\german}{eufm10}

\begin{document}
\pagestyle{plain}

\title
{The structure of the Kac-Wang-Yan algebra}

\author{Andrew R. Linshaw}
\address{Department of Mathematics, University of Denver}
\email{andrew.linshaw@du.edu}
\thanks{This work was partially supported by a grant from the Simons Foundation (\#318755 to Andrew Linshaw)}

{\abstract \noindent The Lie algebra $\cD$ of regular differential operators on the circle has a universal central extension $\hat{\cD}$. The invariant subalgebra $\hat{\cD}^+$ under an involution preserving the principal gradation was introduced by Kac, Wang, and Yan. The vacuum $\hat{\cD}^+$-module with central charge $c\in\mathbb{C}$, and its irreducible quotient $\cV_c$, possess vertex algebra structures, and $\cV_c$ has a nontrivial structure if and only if $c\in \frac{1}{2}\mathbb{Z}$. We show that for each integer $n>0$, $\cV_{n/2}$ and $\cV_{-n}$ are $\cW$-algebras of types $\cW(2,4,\dots,2n)$ and $\cW(2,4,\dots, 2n^2+4n)$, respectively. These results are formal consequences of Weyl's first and second fundamental theorems of invariant theory for the orthogonal group $\text{O}(n)$ and the symplectic group $\text{Sp}(2n)$, respectively. Based on Sergeev's theorems on the invariant theory of $\text{Osp}(1,2n)$ we conjecture that $\cV_{-n + 1/2}$ is of type $\cW(2,4,\dots, 4n^2+8n+2)$, and we prove this for $n=1$. As an application, we show that invariant subalgebras of $\beta\gamma$-systems and free fermion algebras under arbitrary reductive group actions are strongly finitely generated.}

\maketitle
\section{Introduction}

Let $\cD$ denote the Lie algebra of regular differential operators on the circle. It has a universal central extension $\hat{\cD} = \cD \oplus \mathbb{C}\kappa$ which was studied by Kac and Peterson in \cite{KP}. Although $\hat{\cD}$ admits a principal $\mathbb{Z}$-gradation and triangular decomposition, its representation theory is nontrivial because the graded pieces are all infinite-dimensional. The problem of constructing and classifying the {\it quasifinite} irreducible, highest-weight representations (i.e., those with finite-dimensional graded pieces) was solved by Kac and Radul in \cite{KRI}. Explicit constructions of these modules were given in terms of the representation theory of $\widehat{\gg\gl}(\infty,R_m)$, which is a central extension of the Lie algebra of infinite matrices over $R_m = \mathbb{C}[t]/(t^{m+1})$ having only finitely many nonzero diagonal entries. The authors also classified all such $\hat{\cD}$-modules which are unitary.

In \cite{FKRW}, the representation theory of $\hat{\mathcal{D}}$ was developed by Frenkel, Kac, Radul, and Wang from the point of view of vertex algebras. For each $c\in\mathbb{C}$, $\hat{\mathcal{D}}$ admits a module $\mathcal{M}_c$ called the {\it vacuum module}, which is a vertex algebra freely generated by elements $J^l$ of weight $l+1$, for $l\geq 0$. The highest-weight representations of $\hat{\mathcal{D}}$ are in one-to-one correspondence with the highest-weight representations of $\mathcal{M}_c$. The irreducible quotient of $\mathcal{M}_c$ by its maximal graded, proper $\hat{\cD}$-submodule is a simple vertex algebra, and is often denoted by $\mathcal{W}_{1+\infty,c}$. These algebras have been studied extensively in both the physics and mathematics literature; see for example \cite{AFMO,ASV,BN,BS,CTZ,FKRW,KRII}. The above central extension is normalized so that $\cM_c$ is reducible if and only if $c\in \mathbb{Z}$. We have $\cW_{1+\infty,0} \cong \mathbb{C}$, and for a positive integer $n$, $\cW_{1+\infty,n}$ is isomorphic to $\cW(\gg\gl_n)$ \cite{FKRW}. In particular $\cW_{1+\infty,n}$ is of type $\cW(1,2,\dots, n)$; in other words it has a minimal strong generating set consisting of an element in each weight $1,2,\dots, n$. It is known that $\cW_{1+\infty,-1}$ is isomorphic to $\cW(\gg\gl_3)$ \cite{WaI}, but the structure of $\cW_{1+\infty,-n}$ for $n>1$ is more complicated. It is of type $\cW(1,2,\dots, n^2+2n)$ \cite{LI}, but it is not known if it can be related to a standard $\cW$-algebra.

\subsection{The Kac-Wang-Yan algebra}
There are exactly two anti-involutions $\sigma_{\pm}$ of $\hat{\cD}$ (up to conjugation) which preserve the principal $\mathbb{Z}$-gradation. In \cite{KWY}, Kac, Wang, and Yan carried out a similar classification of the quasifinite, irreducible highest-weight modules over the Lie subalgebras $\hat{\cD}^{\pm}$ fixed by $-\sigma_{\pm}$. The authors regarded $\hat{\cD}^+$ as the more fundamental of these two subalgebras, and showed that the vacuum $\hat{\cD}^+$-module $\cM^+_c$ admits a vertex algebra structure for each $c\in\mathbb{C}$. In fact, $\cM^+_c$ is a vertex subalgebra of $\cM_c$, and is freely generated by elements $W^{2m+1}$ of weight $2m+2$, for $m\geq 0$. The unique irreducible quotient of $\cM^+_c$ is denoted by $\cV_c$ in \cite{KWY}, and we shall refer to $\cV_c$ as the {\it Kac-Wang-Yan algebra} in this paper. Let $\pi_c$ denote the projection $\cM^+_c\rightarrow \cV_c$, whose kernel $\mathcal{I}_c$ is the maximal proper graded $\hat{\cD}^+$-submodule of $\cM^+_c$, and let $w^{2k+1} = \pi_c(W^{2k+1})$. The module $\cM^+_c$ is reducible if and only if $c\in \frac{1}{2}\mathbb{Z}$, and $\cV_0 \cong \mathbb{C}$. It was conjectured by Wang in \cite{WaIII} that for a positive integer $n$, $\cV_{n/2}$ is of type $\cW(2,4,\dots, 2n)$, and this was proven for $n=1,2$. By a theorem of de Boer, Feher, and Honecker \cite{BFH}, $\cV_{-1}$ is of type $\cW(2,4,6)$, but for a general $c\in \frac{1}{2} \mathbb{Z}$, finding a minimal strong generating set for $\cV_c$ is an open problem; see Remark 14.6 of \cite{KWY} and Problem 1 of \cite{WaIII}. In this paper, we will solve this problem for negative $c\in \mathbb{Z}$, and for positive $c\in \frac{1}{2} \mathbb{Z}$, and we will give an application of these results to the {\it vertex algebra Hilbert problem}. The remaining case where $c=-n+1/2$ for an integer $n\geq 1$ is more difficult, and we will give a conjectural solution which holds for $n=1$.

\subsection{The case $c=-n$ for $n\geq 1$}
Our starting point is a remarkable free field realization of $\cV_{-n}$ as the $\text{Sp}(2n)$-invariant subalgebra of the $\beta\gamma$-system $\cS(n)$ of rank $n$ \cite{KWY}. The isomorphism $\cV_{-n} \cong \cS(n)^{\text{Sp}(2n)}$ indicates that the structure of $\cV_{-n}$ is deeply connected to classical invariant theory. The action of $\text{Sp}(2n)$ on $\cS(n)$ is analogous to the action of $\text{Sp}(2n)$ on the $n^{\text{th}}$ Weyl algebra $\cD(n)$, and $\text{Sp}(2n)$ is the {\it full} automorphism group of $\cS(n)$. Moreover, $\cS(n)$ admits an $\text{Sp}(2n)$-invariant filtration such that the associated graded algebra $\text{gr}(\cS(n))$ is isomorphic to $\text{Sym} \bigoplus_{k\geq 0} U_k$ as a commutative ring, where each $U_k$ is a copy of the standard $\text{Sp}(2n)$-module $\mathbb{C}^{2n}$. As a vector space, $\cV_{-n}$ is isomorphic to $R = (\text{Sym} \bigoplus_{k\geq 0} U_k)^{\text{Sp}(2n)}$, and we have isomorphisms of graded commutative rings $$\text{gr}(\cV_{-n}) \cong \text{gr}(\cS(n)^{\text{Sp}(2n)}) \cong \text{gr}(\cS(n))^{\text{Sp}(2n)} \cong R.$$ In this sense, we regard $\cV_{-n}$ as a {\it deformation} of the classical invariant ring $R$.

By Weyl's first and second fundamental theorems of invariant theory for the standard representation of $\text{Sp}(2n)$, $R$ is generated by quadratics $$\{q_{a,b}|\ 0\leq a<b\},$$ and the ideal of relations among the $q_{a,b}$'s is generated by the degree $n+1$ Pfaffians $$\{p_I|\ I = (i_0,  \dots, i_{2n+1}),\ 0\leq i_0<\cdots < i_{2n+1}\}.$$ We obtain a corresponding strong generating set $\{\omega_{a,b}\}$ for $\cV_{-n}$, as well as generators $\{P_I\}$ for the ideal of relations among the $\omega_{a,b}$'s, which correspond to the Pfaffians with suitable quantum corrections. The new generating set $\{\omega_{a,b}|\ 0\leq a<b\}$ is related to the old generating set $\{\partial^k w^{2m+1}| k,m\geq 0\}$ by a linear change of variables. The relation of minimal weight occurs at weight $2(n+1)^2$, and corresponds to a singular vector $P_0\in \cI_{-n}\subset \cM^+_{-n}$. 

The technical heart of this paper is the analysis of the quantum corrections of the above classical relations. For each $p_I$, there is a certain correction term $R_I$ appearing in $P_I$ which we call the {\it remainder}. We shall find a closed formula for $R_I$ which implies that the relation in $\cV_{-n}$ of minimal weight has the form \begin{equation} \label{introdecoup} w^{2n^2+4n+1} = Q(w^1, w^3, \dots, w^{2n^2+4n-1}),\end{equation} where $Q$ is a normally ordered polynomial in $w^1, w^3, \dots, w^{2n^2+4n-1}$, and their derivatives. We call \eqref{introdecoup} a {\it decoupling relation}, and by applying the operator $w^3 \circ_1$ repeatedly, we can construct higher decoupling relations $$w^{2m+1} = Q_m(w^1, w^3, \dots, w^{2n^2+4n-1})$$ for all $m>n^2+2n$. This shows that $\{w^1, w^3, \dots, w^{2n^2+4n-1}\}$ is a minimal strong generating set for $\cV_{-n}$, and in particular $\cV_{-n}$ is of type $\cW(2,4, \dots, 2n^2+4n)$; see Theorem \ref{descriptionofvn}.

The proof of this theorem is similar in spirit to our proof in \cite{LI} that $\cW_{1+\infty,-n}$ is of type $\cW(1,2,\dots, n^2+2n)$, but there are some important differences between these two problems. In \cite{LI}, we viewed $\cW_{1+\infty,-n}$ as a deformation of the classical invariant ring $(\text{Sym} \bigoplus_{k\geq 0} (V_k \oplus V^*_k))^{\text{GL}(n)}$, where $V_k$ and $V^*_k$ are isomorphic to the standard $\text{GL}(n)$-module $\mathbb{C}^n$ and its dual, respectively. The generators of this ring are quadratic and the relations are $(n+1)\times (n+1)$ determinants. The main result of \cite{LI} follows from the fact that a certain quantum correction of the relation of minimal weight is nonzero. We proved this inductively without finding an explicit formula for this quantum correction. In the present paper this method does not work, and the only way we know how to prove Theorem \ref{descriptionofvn} is to find a closed formula for $R_I$.

\subsection{The case $c=\frac{n}{2}$ for $n\geq 1$}
There is a free field realization of $\cV_{n/2}$ as the $\text{O}(n)$-invariant subalgebra of the free fermion algebra $\cF(n)$ of rank $n$ \cite{KWY}. The full automorphism group of $\cF(n)$ is $\text{O}(n)$, and there is an $\text{O}(n)$-invariant filtration such that $\text{gr}(\cF(n))\cong \bigwedge \bigoplus_{k\geq 0} U_k$, where each $U_k$ is a copy of the standard $\text{O}(n)$-module $\mathbb{C}^{n}$. As vector spaces, $\cV_{n/2}\cong R = (\bigwedge \bigoplus_{k\geq 0} U_k)^{\text{O}(n)}$, and we have isomorphisms of graded commutative rings $$\text{gr}(\cV_{n/2}) \cong \text{gr}(\cF(n)^{\text{O}(n)}) \cong \text{gr}(\cF(n))^{\text{O}(n)} \cong R.$$ 
The generators of $R$ are quadratics $\{q_{a,b}|\ 0\leq a<b\}$, and the ideal of relations is generated by degree $n+1$ polynomials $$\{d_{I,J}|\ I= (i_0, \dots, i_{n}),\ J = (j_0,  \dots, j_{n}),\ 0\leq i_0 \leq \cdots \leq i_{n},\ 0\leq j_0\leq \cdots \leq j_{n}\},$$ which are odd analogues of determinants. The relation of minimal weight occurs at weight $2n+2$, and by analyzing the quantum corrections of the classical relations, we show that it has the form $w^{2n+1} = Q(w^1, w^3,  \dots, w^{2n-1})$. This gives rise to relations $w^{2m+1} = Q_m(w^1, w^3,  \dots, w^{2n-1})$ for all $m>n$. Therefore $\{w^1, w^3,  \dots, w^{2n-1}\}$ is a minimal strong generating set for $\cV_{n/2}$, so $\cV_{n/2}$ is of type $\cW(2,4, \dots, 2n)$ as conjectured by Wang; see Theorem \ref{descriptionofvnodd}.

\subsection{The case $c = -n+\frac{1}{2}$ for $n\geq 1$}
The remaining case where the central charge is $-n+\frac{1}{2}$ for an integer $n\geq 1$ is more difficult. There is a free field realization of $\cV_{-n+1/2}$ as the $\text{Osp}(1,2n)$-invariant subalgebra of $\cS(n) \otimes \cF(1)$, and $\text{gr}(\cS(n)\otimes \cF(1))$ is isomorphic to $\text{Sym} \bigoplus_{k\geq 0} U_k$ as a commutative ring. Here $U_k$ is a copy of the standard $\text{Osp}(1,2n)$-module $\mathbb{C}^{2n|1}$. As a vector space, $\cV_{-n+1/2 }\cong R = (\text{Sym} \bigoplus_{k\geq 0} U_k)^{\text{Osp}(1,2n)}$, and we have isomorphisms of graded commutative rings $$\text{gr}(\cV_{-n+1/2}) \cong \text{gr}((\cS(n) \otimes \cF(1))^{\text{Osp}(1,2n)}) \cong \text{gr}(\cS(n) \otimes \cF(1))^{\text{Osp}(1,2n)} \cong R.$$ The generators and relations for $R$ were given by Sergeev \cite{SI,SII}, and the relation of minimal weight occurs at weight $4n^2+8n+4$. We conjecture that it has the form \begin{equation} \label{introdecoupsergeev} w^{4n^2+8n+3} = Q(w^1, w^3,  \dots, w^{4n^2+8n+1}).\end{equation} If this is correct we can construct relations $w^{2m+1} = Q_m(w^1, w^3,  \dots, w^{4n^2+8n+1})$ for all $m>2n^2+4n+1$. This would imply that $\cV_{-n+1/2}$ is of type $\cW(2,4,\dots, 4n^2+8n+2)$, but unfortunately the quantum corrections are more difficult to analyze in this case. A computer calculation shows that our conjecture holds for $n=1$, so $\cV_{-1/2}$ is of type $\cW(2,4,\dots, 14)$.

\subsection{Representation theory of $\cV_{-n}$ and $\cV_{n/2}$}
The representation theory of $\cV_{-n}$ is governed by its Zhu algebra $A(\cV_{-n})$, which is a commutative algebra on generators $a^1, a^3, \dots, a^{2n^2+4n-1}$, corresponding to the generators of $\cV_{-n}$. The irreducible, admissible $\mathbb{Z}_{\geq 0}$-graded $\cV_{-n}$-modules are therefore all highest-weight modules, and are parametrized by the points in the variety $\text{Spec}(A(\cV_{-n}))$, which is a proper, closed subvariety of $\mathbb{C}^{n^2+2n}$. Similarly, the Zhu algebra of $\cV_{n/2}$ is a commutative algebra on generators $a^1, a^3, \dots, a^{2n-1}$. The irreducible, admissible $\mathbb{Z}_{\geq 0}$-graded $\cV_{n/2}$-modules are also highest-weight modules, and are parametrized by a proper, closed subvariety of $\mathbb{C}^{n}$. 

$\text{Sp}(2n)$ and $\cV_{-n}$ form a dual reductive pair acting on $\cS(n)$ in the sense that \begin{equation} \label{kwydec} \cS(n) \cong \bigoplus_{\nu\in H} L(\nu)\otimes M^{\nu},\end{equation} where $H$ indexes the irreducible, finite-dimensional representations $L(\nu)$ of $\text{Sp}(2n)$, and the $M^{\nu}$'s are inequivalent, irreducible, highest-weight $\cV_{-n}$-modules. Likewise, \begin{equation} \label{kwydecodd} \cF(n) \cong \bigoplus_{\mu\in H'} L(\mu)\otimes M^{\mu},\end{equation} where $H'$ indexes the irreducible, finite-dimensional representations $L(\mu)$ of $\text{O}(n)$, and the $M^{\mu}$'s are inequivalent, irreducible, highest-weight $\cV_{n/2}$-modules. The modules $M^{\nu}$ and $M^{\mu}$ above correspond to rational points on $\text{Spec}(A(\cV_{-n}))$ and $\text{Spec}(A(\cV_{n/2}))$, respectively, and they have the $C_1$-cofiniteness property according to Miyamoto's definition \cite{M}. This is essential for our applications.

\subsection{The vertex algebra Hilbert problem} 

A vertex algebra $\cA$ is called {\it strongly finitely generated} if there exists a finite set of generators such that the set of iterated Wick products of the generators and their derivatives spans $\cA$. This property has many important consequences, and in particular implies that the Zhu algebra of $\cA$ is finitely generated. Recall Hilbert's theorem that if a reductive group $G$ acts on a finite-dimensional complex vector space $V$, the invariant ring $\cO(V)^G$ is finitely generated \cite{HI,HII}. This theorem was very influential in the development of commutative algebra and algebraic geometry. In fact, Hilbert's basis theorem, Nullstellensatz, and syzygy theorem were all introduced in connection with this problem. One can ask similar questions in the setting of noncommutative rings such as Weyl algebras or universal enveloping algebras. In these cases, there are nice filtrations allowing the problem to be reduced to commutative algebra. The analogous problem for vertex algebras is the following. 

\begin{prob} Given a simple, strongly finitely generated vertex algebra $\cA$ and a reductive group $G$ of automorphisms of $\cA$, is $\cA^G$ strongly finitely generated?
\end{prob}

This is much more subtle than the case of noncommutative rings, and generally fails for abelian vertex algebras. The main difficulty is that vertex algebras are not Noetherian, and this phenomenon depends sensitively on the {\it nonassociativity} of vertex algebras. Isolated examples have been known for many years in both the physics and mathematics literature (see for example \cite{BFH,B-H,EFH,DN,FKRW,KWY}), although the first general results of this kind were obtained in \cite{LII}, in the case where $G\subset \text{GL}(n)$ and $\cA$ is either the $\beta\gamma$-system $\cS(n)$, or the $bc$-system $\cE(n)$, which is isomorphic to $\cF(2n)$. We have isomorphisms $\cS(n)^{\text{GL}(n)} \cong \cW_{1+\infty,-n}$ and $\cE(n)^{\text{GL}(n)} \cong \cW_{1+\infty,n}$ by results in \cite{KRII} and \cite{FKRW}, respectively. For $G\subset \text{GL}(n)$, $\cS(n)^G$ and $\cE(n)^G$ are completely reducible as modules over $\cW_{1+\infty,-n}$ and $\cW_{1+\infty,n}$, respectively. The strong finite generation of $\cS(n)^G$ and $\cE(n)^G$ can be deduced from these decompositions, together with the structure of $\cW_{1+\infty,\pm n}$.

For a general reductive $G\subset \text{Sp}(2n)$, the structure of $\cS(n)^G$ is governed by $\cV_{-n}$ rather than $\cW_{1+\infty,-n}$. Similarly, for $G\subset \text{O}(n)$, the structure of $\cF(n)^G$ is governed by $\cV_{n/2}$. Using the structure of $\cV_{-n}$ and $\cV_{n/2}$ and the complete reducibility of $\cS(n)^G$ and $\cF(n)^G$ as modules over $\cV_{-n}$ and $\cV_{n/2}$, respectively, we will show that $\cS(n)^G$ and $\cF(n)^G$ are strongly finitely generated; see Theorems \ref{sfg} and \ref{sfgodd}. Our proof is essentially constructive, and it provides a complete solution to the Hilbert problem for $\cS(n)$ and $\cF(n)$.

\section{Vertex algebras}

We will assume that the reader is familiar with the basics of vertex algebra theory, which has been discussed from various points of view in the literature (see for example \cite{B,FLM,FHL,K,FBZ}). We will follow the formalism developed in \cite{LZ} and partly in \cite{LiI}, and we will use the notation of our previous paper \cite{LI}. By a {\it vertex algebra}, we mean a quantum operator algebra $\cA$ in which any two elements $a,b$ are local, meaning that $(z-w)^N[a(z), b(w)] = 0$ for some positive integer $N$. Here $\cA$ is assumed to be $\mathbb{Z}/2\mathbb{Z}$-graded, and $[,]$ denotes the super bracket. This is well known to be equivalent to the notion of vertex algebra in \cite{FLM}. The operators product expansion (OPE) formula is given by
$$a(z)b(w)\sim\sum_{n\geq 0}a(w)\circ_n b(w)\ (z-w)^{-n-1}.$$ Here $\sim$ means equal modulo terms which are regular at $z=w$, and $\circ_n$ denotes the $n^{\text{th}}$ circle product. A subset $S=\{a_i|\ i\in I\}$ of $\cA$ is said to {\it generate} $\cA$ if $\cA$ is spanned by words in the letters $a_i$, $\circ_n$, for $i\in I$ and $n\in\mathbb{Z}$. We say that $S$ {\it strongly generates} $\cA$ if $\cA$ is spanned by words in the letters $a_i$, $\circ_n$ for $n<0$. Equivalently, $\cA$ is spanned by $$\{ :\partial^{k_1} a_{i_1}\cdots \partial^{k_m} a_{i_m}:| \ i_1,\dots,i_m \in I,\ k_1,\dots,k_m \geq 0\}.$$ We say that $S$ {\it freely generates} $\cA$ if there are no nontrivial normally ordered polynomial relations among the generators and their derivatives.

\section{Category $\mathcal{R}$}
Let $\cR$ be the category of vertex algebras $\cA$ equipped with a $\mathbb{Z}_{\geq 0}$-filtration
\begin{equation} \cA_{(0)}\subset\cA_{(1)}\subset\cA_{(2)}\subset \cdots,\qquad \cA = \bigcup_{d\geq 0}
\cA_{(d)}\end{equation} such that $\cA_{(0)} = \mathbb{C}$, and for all
$a\in \cA_{(r)}$, $b\in\cA_{(s)}$, we have
\begin{equation} \label{goodi} a\circ_n b\in  \bigg\{\begin{matrix}\cA_{(r+s)} & n<0 \\ \cA_{(r+s-1)} & 
n\geq 0 \end{matrix}\ . \end{equation}
Elements $a(z)\in\cA_{(d)}\setminus \cA_{(d-1)}$ are said to have degree $d$.

Filtrations on vertex algebras satisfying \eqref{goodi} are known as {\it good increasing filtrations} \cite{LiII}. Setting $\cA_{(-1)} = \{0\}$, the associated graded algebra $\text{gr}(\cA) = \bigoplus_{d\geq 0}\cA_{(d)}/\cA_{(d-1)}$ is a $\mathbb{Z}_{\geq 0}$-graded associative, (super)commutative algebra with a
unit $1$ under a product induced by the Wick product on $\cA$. For $r\geq 1$ we have the projection \begin{equation} \phi_r: \cA_{(r)} \ra \cA_{(r)}/\cA_{(r-1)}\subset \text{gr}(\cA).\end{equation} 
The operator $\partial$ on $\cA$ induces a derivation of degree zero on $\text{gr}(\cA)$ which we also denote by $\partial$. For each $a\in\cA_{(d)}$ and $n\geq 0$, the operator $a\circ_n$ on $\cA$
induces a derivation of degree $d-k$ on $\text{gr}(\cA)$, which we denote by $a(n)$. Here $$k  = \text{sup} \{ j\geq 1|\ \cA_{(r)}\circ_n \cA_{(s)}\subset \cA_{(r+s-j)},\ \forall r,s,n\geq 0\},$$ as in \cite{LL}. These derivations give $\text{gr}(\cA)$ a vertex Poisson algebra structure.

The assignment $\cA\mapsto \text{gr}(\cA)$ is a functor from $\cR$ to the category of $\mathbb{Z}_{\geq 0}$-graded (super)commutative rings with a differential $\partial$ of degree zero, which we call $\partial$-rings. A $\partial$-ring is just an {\it abelian} vertex algebra, that is, a vertex algebra $\cV$ in which $[a(z),b(w)] = 0$ for all $a,b\in\cV$. A $\partial$-ring $A$ is said to be generated by a set $\{a_i|\ i\in I\}$ if $\{\partial^k a_i|\ i\in I, k\geq 0\}$ generates $A$ as a ring. The key feature of $\cR$ is the following reconstruction property \cite{LL}.

\begin{lemma}\label{reconlem}Let $\cA$ be a vertex algebra in $\cR$ and let $\{a_i|\ i\in I\}$ be a set of generators for $\text{gr}(\cA)$ as a $\partial$-ring, where $a_i$ is homogeneous of degree $d_i$. If $a_i(z)\in\cA_{(d_i)}$ are elements satisfying $\phi_{d_i}(a_i(z)) = a_i$, then $\cA$ is strongly generated as a vertex algebra by $\{a_i(z)|\ i\in I\}$.\end{lemma}

There is a similar reconstruction property for kernels of surjective morphisms \cite{LI}. Let $f:\cA\rightarrow \cB$ be a morphism in $\cR$ with kernel $\cJ$, such that $f$ maps each $\cA_{(d)}$ onto $\cB_{(d)}$. The kernel $J$ of the induced map $\text{gr}(f): \text{gr}(\cA)\rightarrow \text{gr}(\cB)$ is a homogeneous $\partial$-ideal, i.e., $\partial J \subset J$. A set $\{a_i|\ i\in I\}$ such that $a_i$ is homogeneous of degree $d_i$ is said to generate $J$ as a $\partial$-ideal if $\{\partial^k a_i|\ i\in I,\ k\geq 0\}$ generates $J$ as an ideal.

\begin{lemma} \label{idealrecon} Let $\{a_i| i\in I\}$ be a generating set for $J$ as a $\partial$-ideal, where $a_i$ is homogeneous of degree $d_i$. Then there exist elements $a_i(z)\in \cA_{(d_i)}$ with $\phi_{d_i}(a_i(z)) = a_i$, such that $\{a_i(z)|\ i\in I\}$ generates $\cJ$ as a vertex algebra ideal.\end{lemma}

\section{The vertex algebra $\cW_{1+\infty,c}$}
Let $\cD$ be the Lie algebra of regular differential operators on the circle, with coordinate $t$. A standard basis for $\cD$ is $$J^l_k = -t^{l+k} (\partial_t)^l,\qquad k\in \mathbb{Z},\qquad l\in \mathbb{Z}_{\geq 0},$$ where $\partial_t = \frac{d}{dt}$. $\cD$ has a $2$-cocycle given by \begin{equation}\label{cocycle} \Psi\big(f(t) (\partial_t)^m,  g(t) (\partial_t)^n\big) = \frac{m! n!}{(m+n+1)!} \text{Res}_{t=0} f^{(n+1)}(t) g^{(m)}(t) dt,\end{equation} and a corresponding central extension $\hat{\cD} = \cD \oplus \mathbb{C} \kappa$ \cite{KP}. $\hat{\cD}$ has a $\mathbb{Z}$-grading $\hat{\cD} = \bigoplus_{j\in\mathbb{Z}} \hat{\cD}_j$ by weight, given by
$$\text{wt} (J^l_k) = k,\qquad \text{wt} (\kappa) = 0,$$ and a triangular decomposition $\hat{\cD} = \hat{\cD}_+\oplus\hat{\cD}_0\oplus \hat{\cD}_-$, where $\hat{\cD}_{\pm} = \bigoplus_{j\in \pm \mathbb{N}} \hat{\cD}_j$ and $\hat{\cD}_0 = \cD_0\oplus \mathbb{C}\kappa$. For $c\in\mathbb{C}$ and $\lambda\in \cD_0^*$, define the Verma module 
$$\cM_c(\hat{\cD},\lambda) = U(\hat{\cD})\otimes_{U(\hat{\cD}_0\oplus \hat{\cD}_+)} \C_{\lambda},$$ where $\C_{\lambda}$ is the one-dimensional $\hat{\cD}_0\oplus \hat{\cD}_+$-module on which $\kappa$ acts by multiplication by $c$ and $h\in\hat{\cD}_0$ acts by multiplication by $\lambda(h)$, and $\hat{\cD}_+$ acts by zero. Let $\cP$ be the parabolic subalgebra of $\cD$ consisting of differential operators which extend to all of $\mathbb{C}$, which has a basis $\{J^l_k|\ l\geq 0,\ l+k\geq 0\}$. The cocycle $\Psi$ vanishes on $\cP$, so $\cP$ may be regarded as a subalgebra of $\hat{\cD}$, and $\hat{\cD}_0\oplus \hat{\cD}_+\subset \hat{\cP}$, where $\hat{\cP} = \cP\oplus \mathbb{C}\kappa$. The induced $\hat{\cD}$-module $$\cM_c = U(\hat{\cD})\otimes_{U(\hat{\cP})} \C_0$$ is a quotient of $\cM_c(\hat{\cD},0)$, and is known as the {\it vacuum $\hat{\cD}$-module of central charge $c$}. $\cM_c$ has a vertex algebra structure with generators $$J^l(z) = \sum_{k\in\mathbb{Z}} J^l_k z^{-k-l-1},\qquad l\geq 0,$$ of weight $l+1$. The modes $J^l_k$ represent $\hat{\cD}$ on $\cM_c$, and we write these elements in the form
$$J^l(z) = \sum_{k\in\mathbb{Z}} J^l(k) z^{-k-1},$$ where $J^l(k) = J^l_{k-l}$. In fact, $\cM_c$ is freely generated by $\{J^l(z)|\ l\geq 0\}$; the monomials $ :\partial^{i_1}J^{l_1}(z)\cdots \partial^{i_r} J^{l_r}(z):$ such that $l_1\leq \cdots \leq l_r$ and $i_a\leq i_b$ if $l_a = l_b$, form a basis for $\cM_c$. 

A weight-homogeneous element $\omega\in \mathcal{M}_c$ is called a {\it singular vector} if $J^l\circ_k \omega = 0$ for all $k>l\geq 0$. The maximal proper $\hat{\mathcal{D}}$-submodule $\mathcal{I}_c$ is the vertex algebra ideal generated by all singular vectors $\omega\neq 1$, and the unique irreducible quotient $\mathcal{M}_c/\mathcal{I}_c$ is denoted by $\mathcal{W}_{1+\infty,c}$. We denote the image of $J^l$ in $\cW_{1+\infty,c}$ by $j^l$. The cocycle \eqref{cocycle} is normalized so that $\cM_c$ is reducible if and only if $c\in \mathbb{Z}$. For a positive integer $n$ there is an isomorphism $\mathcal{W}_{1+\infty,n} \cong \mathcal{W}(\mathfrak{g}\mathfrak{l}_n)$, and in particular $\cW_{1+\infty,n}$ has a minimal strong generating set $\{j^0, j^1,\dots, j^{n-1}\}$ \cite{FKRW}. It is also known that $\cW_{1+\infty,-1}\cong \cW(\gg\gl_3)$, and is strongly generated by $\{j^0, j^1, j^2\}$ \cite{WaI}. This result was generalized in \cite{LI}; for all $n\geq 1$, $\cW_{1+\infty,-n}$ has a minimal strong generating set $\{j^0, j^1, \dots, j^{n^2+2n-1}\}$.

\section{The vertex algebra $\cV_c$}
The Lie algebra $\cD$ has an anti-involution $\sigma_{+,-1}$ given by $\sigma_{+,-1}(t) = t$ and $\sigma_{+,-1}(\partial_t) = -\partial_t$ \cite{KWY}. The subalgebra $\cD^+$ fixed by $-\sigma_{+,-1}$ has generators $$W^m_k = -\frac{1}{2}\big(t^{k+m}(\partial_t)^m + (-1)^{m+1} (\partial_t)^m t^{k+m}\big),\qquad k\in \mathbb{Z},\qquad m\in 1+ 2\mathbb{Z}_{\geq 0}.$$ Note that $\{W^1_k| k\in\mathbb{Z}\}$ spans a copy of the Virasoro Lie algebra. We use the same notation $\Psi$ to denote the restriction of the cocycle $\Psi$ to $\cD^+$. Let $\hat{\cD}^+$ be the corresponding central extension of $\cD^+$, which is a subalgebra of $\hat{\cD}$. Let $\cP^+ = \cP \cap \cD^+$, which is a subalgebra of $\hat{\cD}^+$ since $\Psi$ vanishes on $\cP^+$. Clearly $\cP^+$ has a basis $$\{ W^m_k| m+k \geq 0,\  k\in \mathbb{Z}, \ m\in 1+ 2\mathbb{Z}_{\geq 0} \}.$$ The induced $\hat{\cD}^+$-module $\cM_c^+ = U(\hat{\cD}^+)\otimes_{U(\hat{\cP}^+)} \C_0$ is known as the {\it vacuum $\hat{\cD}^+$-module of central charge $c$}. It has a vertex algebra structure with generators $$W^m(z) = \sum_{k\in\mathbb{Z}} W^m_k z^{-k-m-1}, \qquad m\in 1+2\mathbb{Z}_{\geq 0} $$ of weight $m+1$. The modes $W^m_k$ represent $\hat{\cD}^+$ on $\cM^+_c$, and we write
$$W^m(z) = \sum_{k\in\mathbb{Z}} W^m(k) z^{-k-1},$$ where $W^m(k) = W^m_{k-m}$. Clearly $\cM^+_c$ is a vertex subalgebra of $\cM_c$, and $\cM^+_c$ is freely generated by $\{W^m(z)|\ m\in 1+ 2\mathbb{Z}_{\geq 0} \}$; the monomials \begin{equation}\label{standmon} :\partial^{i_1}W^{m_1}(z)\cdots \partial^{i_r} W^{m_r}(z):,\end{equation} such that $m_1\leq \cdots \leq m_r$ and $i_a\leq i_b$ if $m_a = m_b$, form a basis for $\cM^+_c$.

Define a filtration $$(\mathcal{M}^+_c)_{(0)} \subset (\mathcal{M}^+_c)_{(1)}\subset \cdots$$ on $\mathcal{M}^+_c$ as follows: for $d\geq 0$, $(\mathcal{M}^+_c)_{(2d)}$ is spanned by monomials (\ref{standmon}) for $r\leq d$, and $(\mathcal{M}^+_c)_{(2d+1)} = (\mathcal{M}^+_c)_{(2d)}$. In particular, $\partial^k W^{2m+1}$ has degree $2$ for all $k,m\geq 0$. Equipped with this filtration, $\mathcal{M}^+_c$ lies in the category $\mathcal{R}$, and $\text{gr}(\mathcal{M}^+_c)$ is the polynomial algebra $\mathbb{C}[\partial^k W^{2m+1}|\  k,m\geq 0]$. For $k,m\geq 0$, each $W^{2m+1}(k) \in \cP^+$ gives rise to a derivation of degree zero on $\text{gr}(\mathcal{M}^+_c)$ coming from the vertex Poisson algebra structure, and this action of $\mathcal{P}^+$ on $\text{gr}(\mathcal{M}^+_c)$ is independent of $c$.

\begin{lemma} \label{weakfg} For each $c\in\mathbb{C}$, $\cM^+_c$ is generated as a vertex algebra by $W^3$. \end{lemma}

\begin{proof} First, $W^1 = \frac{1}{360} W^3 \circ_5 W^3$, so $W^1$ lies in the subalgebra $\bra W^3 \ket$ generated by $W^3$. Next, an OPE calculation shows that for all $m>0$, \begin{equation} \label{raisingoperator} W^3 \circ_1 W^{2m+1} \equiv -(2m+2) W^{2m+3},\end{equation} modulo a linear combination of terms of the form $\partial^{2k} W^{2m+3 -2k}$ for $1\leq k\leq m+1$. It follows by induction on $m$ that each $W^{2m+1} \in \bra W^3 \ket$. \end{proof}

In particular, $\cM^+_c$ is a finitely generated vertex algebra. However, since $\cM^+_c$ is freely generated by $\{W^{2m+1}|\ m\geq 0\}$, $\cM^+_c$ is not {\it strongly} generated by any finite set. A weight-homogeneous element $\omega\in \cM^+_c$ is called a {\it singular vector} if $W^{2m+1}\circ_k \omega = 0$ for all $m\geq 0$ and $k>2m+1$. The maximal proper $\hat{\cD}^+$-submodule $\cI_c$ is the vertex algebra ideal generated by all singular vectors $\omega\neq 1$, and the Kac-Wang-Yan algebra $\cV_c$ is the unique irreducible quotient $\cM^+_c/\cI_c$. We denote the projection $\cM^+_c\ra \cV_c$ by $\pi_{c}$, and we use the notation \begin{equation}\label{lowercasej} w^{2m+1} = \pi_c (W^{2m+1}),\qquad m\geq 0.\end{equation} Clearly $\cV_c$ is generated as a vertex algebra by $w^3$. Since $\cM^+_c$ is reducible for $c\in \frac{1}{2}\mathbb{Z}$, there exist normally ordered polynomial relations among $\{w^{2m+1}|\ m\geq 0\}$ and their derivatives in this case. In \cite{WaIII}, Wang showed that $\cV_{1/2}$ and $\cV_1$ have minimal strong generating sets $\{w^1\}$ and $\{w^1, w^3\}$, respectively, and he conjectured that for each integer $n>1$, $\cV_{n/2}$ has a minimal strong generating set $\{w^1, w^3, \dots, w^{2n-1}\}$. It is also known that $\cV_{-1}$ has a minimal strong generating set $\{w^1, w^3, w^5\}$ \cite{BFH}, but to the best of our knowledge there are no other results or conjectures for negative $c\in \frac{1}{2}\mathbb{Z}$.

\section{The case $c=-n$ for $n\geq 1$}
For each integer $n\geq 1$, $\cV_{-n}$ admits a free field realization as the $\text{Sp}(2n)$-invariant subalgebra of the rank $n$ $\beta\gamma$-system $\cS(n)$, which was denoted by $\cF^{\otimes -n}$ in \cite{KWY}. The $\beta\gamma$-system was introduced in \cite{FMS}, and has even generators $\beta^{i}$, $\gamma^{i}$ for $i=1,\dots, n$, which satisfy $$\beta^i(z)\gamma^{j}(w)\sim \delta_{i,j} (z-w)^{-1},\qquad \gamma^{i}(z)\beta^j(w)\sim -\delta_{i,j} (z-w)^{-1},$$ \begin{equation}\label{betagamma} \beta^i(z)\beta^j(w)\sim 0,\qquad \gamma^i(z)\gamma^j (w)\sim 0.\end{equation}
We give $\cS(n)$ the conformal structure 
\begin{equation}\label{oneparameter} L = \frac{1}{2} \sum_{i=1}^n \big(:\beta^{i}\partial\gamma^{i}: - :\partial\beta^{i}\gamma^{i}:\big)\end{equation} of central charge $-n$, under which $\beta^{i}$, $\gamma^{i}$ are primary of weight $\frac{1}{2}$. The full automorphism group of $\cS(n)$ is $\text{Sp}(2n)$, and $\{\beta^{i}, \gamma^{i} | \ i=1,\dots,n\}$ spans a copy of the standard $\text{Sp}(2n)$-module $\mathbb{C}^{2n}$. There is a basis of $\cS(n)$ consisting of normally ordered monomials
\begin{equation}\label{basisofs} :\partial^{I_1} \beta^{1}\cdots \partial^{I_n} \beta^{n}\partial^{J_1} \gamma^{1}\cdots\partial^{J_n} \gamma^{n}: .\end{equation} In this notation, $I_k = (i^k_1,\dots, i^k_{r_k})$ and $J_k =(j^k_1,\dots, j^k_{s_k})$ are lists of integers satisfying $0\leq i^k_1\leq \cdots \leq i^k_{r_k}$ and  $0\leq j^k_1\leq \cdots \leq j^k_{s_k}$, and 
$$\partial^{I_k}\beta^{k} = \ :\partial^{i^k_1}\beta^{k} \cdots \partial^{i^k_{r_k}} \beta^{k}:,\qquad \partial^{J_k}\gamma^{k} = \ :\partial^{j^k_1}\gamma^{k} \cdots \partial^{j^k_{s_k}} \gamma^{k}:.$$ We have a $\mathbb{Z}_{\geq 0}$-grading \begin{equation}\label{grading} \cS(n) = \bigoplus_{d\geq 0} \cS(n)^{(d)},\end{equation} where $\cS(n)^{(d)}$ is spanned by monomials \eqref{basisofs} of total degree $ d = \sum_{k=1}^n r_k + s_k$. Finally, we define the filtration $\cS(n)_{(d)} = \bigoplus_{i=0}^d \cS(n)^{(i)}$. This filtration satisfies \eqref{goodi} and we have an isomorphism of $\text{Sp}(2n)$-modules
\begin{equation} \label{linassgrad} \cS(n) \cong  \text{gr}(\cS(n)),\end{equation}
and an isomorphism of graded commutative rings
\begin{equation} \label{genassgrad} \text{gr}(\cS(n))\cong \text{Sym} \bigoplus_{k\geq 0} U_k.\end{equation} Here $U_k$ is the copy of the standard $\text{Sp}(2n)$-module $\mathbb{C}^{2n}$ spanned by $\{\beta^{i}_k, \gamma^{i}_k |\ i=1,\dots, n\}$. In this notation, $\beta^{i}_k$ and $\gamma^{i}_k$ are the images of $\partial^k \beta^{i}(z)$ and $\partial^k\gamma^{i}(z)$ in $\text{gr}(\cS(n))$. The following result appears in Proposition 14.2 of \cite{KWY}.
\begin{thm} There is an isomorphism $\cV_{-n}\rightarrow \cS(n)^{\text{Sp}(2n)}$ given by
\begin{equation} \label{bgrealization} w^{2m+1} \mapsto \frac{1}{2} \sum_{i=1}^n \big(:\beta^{i}\partial^{2m+1} \gamma^{i}:- :\partial^{2m+1} \beta^{i} \gamma^{i}: \big),\qquad m\geq 0.\end{equation} 
\end{thm}
Since the action of $\text{Sp}(2n)$ on $\cS(n)$ preserves the grading \eqref{grading}, $\cV_{-n}$ is a graded subalgebra of $\cS(n)$. We write \begin{equation}\label{gradingv} \cV_{-n} = \bigoplus_{d\geq 0} ( \cV_{-n})^{(d)},\qquad (\cV_{-n})^{(d)} = \cV_{-n}\cap\cS(n)^{(d)},\end{equation} and define the corresponding filtration by $(\cV_{-n} )_{(d)} = \bigoplus_{i=0}^{d} (\cV_{-n} )^{(i)}$. Clearly \eqref{bgrealization} preserves conformal structures and is a morphism in the category $\cR$. 

The identification $\cV_{-n}\cong \cS(n)^{\text{Sp}(2n)}$ suggests an alternative strong generating set for $\cV_{-n}$ coming from classical invariant theory. We have \begin{equation} \label{grisos} \text{gr}(\cV_{-n}) \cong \text{gr}(\cS(n)^{\text{Sp}(2n)}) \cong \text{gr}(\cS(n))^{\text{Sp}(2n)} \cong (\text{Sym}\bigoplus_{k\geq 0} U_k)^{\text{Sp}(2n)}.\end{equation} The generators and relations for $(\text{Sym} \bigoplus_{k\geq 0} U_k )^{\text{Sp}(2n)}$ are given by Weyl's {\it first and second fundamental theorems of invariant theory} for the standard representation of $\text{Sp}(2n)$ (Theorems 6.1.A and 6.1.B of \cite{We}).

\begin{thm} \label{weylfft} For $k\geq 0$, let $U_k$ be the copy of the standard $\text{Sp}(2n)$-module $\mathbb{C}^{2n}$ with symplectic basis $\{x_{i,k}, y_{i,k}| \ i=1,\dots,n\}$. Then $(\text{Sym} \bigoplus_{k\geq 0} U_k )^{\text{Sp}(2n)}$ is generated by the quadratics \begin{equation}\label{weylgenerators} q_{a,b} = \frac{1}{2}\sum_{i=1}^n \big( x_{i,a} y_{i,b} - x_{i,b} y_{i,a}\big),\qquad a,b \geq 0. \end{equation} Note that $q_{a,b} = -q_{b,a}$. Let $\{Q_{a,b}|\ a,b\geq 0\}$ be commuting indeterminates satisfying $Q_{a,b} = -Q_{b,a}$. The kernel $I_n$ of the homomorphism \begin{equation}\label{weylquot} \mathbb{C}[Q_{a,b}]\ra (\text{Sym} \bigoplus_{k\geq 0} U_k)^{\text{Sp}(2n)},\qquad Q_{a,b}\mapsto q_{a,b},\end{equation} is generated by the degree $n+1$ Pfaffians $p_I$, which are indexed by lists $I = (i_0,\dots,i_{2n+1})$ of integers satisfying \begin{equation}\label{ijineq} 0\leq i_0<\cdots < i_{2n+1}.\end{equation} For $n=1$ and $I = (i_0, i_1, i_2, i_3)$, we have $$p_I = q_{i_0, i_1} q_{i_2, i_3} - q_{i_0, i_2} q_{i_1, i_3}+q_{i_0, i_3} q_{i_1, i_2},$$ and for $n>1$ they are defined inductively by \begin{equation} \label{pfaffinduction} p_I =  \sum_{r=1}^{2n+1} (-1)^{r+1} q_{i_0,i_r} p_{I_r},\end{equation} where $I_r = (i_1,\dots, \widehat{i_r},\dots, i_{2n+1})$ is obtained from $I$ by omitting $i_0$ and $i_r$. \end{thm}

Under \eqref{grisos}, the generators $q_{a,b}$ correspond to strong generators
\begin{equation}\label{newgenomega} \omega_{a,b} = \frac{1}{2}\sum_{i=1}^n \big(:\partial^a\beta^{i} \partial ^b \gamma^{i}: -:\partial^b\beta^{i} \partial ^a \gamma^{i}:\big)\end{equation} of $\cV_{-n}$. In this notation, $w^{2m+1} = \omega_{0,2m+1}$ for $m\geq 0$. Let $A_m$ denote the vector space spanned by $\{\omega_{a,b}|\  a+b = m\}$. Clearly $A_1$ is spanned by $w^1$ and for $m>0$, \begin{equation}\label{decompofa} A_{2m+1} = \partial (A_{2m})\oplus \bra w^{2m+1}\ket =  \partial^2 (A_{2m-1})\oplus \bra w^{2m+1}\ket ,\end{equation} where $\bra w^{2m+1}\ket$ is the linear span of $w^{2m+1}$. Similarly, $A_2$ is spanned by $\partial w^1$ and for $m>0$, \begin{equation}\label{decompofai} A_{2m+2} = \partial^2(A_{2m})\oplus \bra \partial w^{2m+1}\ket =  \partial^3 (A_{2m-1})\oplus \bra \partial w^{2m+1}\ket.\end{equation}  It is easy to see that $\{\partial^{2i} w^{2m+1-2i}|\  0\leq i\leq m\}$ and $\{\partial^{2i+1} w^{2m+1-2i}|\ 0\leq i\leq m\}$ are bases of $A_{2m+1}$ and $A_{2m+2}$, respectively. For $a+b = 2m+1$ and $c+d = 2m+2$, $\omega_{a,b}$ and $\omega_{c,d}$ can be expressed uniquely in the form \begin{equation}\label{lincomb} \omega_{a,b} =\sum_{i=0}^m \lambda_i \partial^{2i}w^{2m+1-2i},\qquad \omega_{c,d} =\sum_{i=0}^m \mu_i \partial^{2i+1}w^{2m+1-2i}\end{equation} for constants $\lambda_i,\mu_i$. Hence $\{\partial^k w^{2m+1}|\ k, m\geq 0\}$ and $\{\omega_{a,b}|\ 0\leq a < b\}$ are related by a linear change of variables. Using (\ref{lincomb}), which holds in $\cV_{-n}$ for all $n$, we can define an alternative strong generating set $\{\Omega_{a,b}| \ 0\leq a< b\}$ for $\cM^+_{-n}$ by the same formula: for $a+b = 2m+1$ and $c+d = 2m+2$,
$$\Omega_{a,b} =\sum_{i=0}^m \lambda_i \partial^{2i}W^{2m+1-2i},\qquad  \Omega_{c,d} =\sum_{i=0}^m \mu_i \partial^{2i+1}W^{2m+1-2i}.$$ Clearly $\pi_{-n}(\Omega_{a,b}) = \omega_{a,b}$. We also denote the span of $\{\Omega_{a,b}|\ a+b = m\}$ by $A_m$ when no confusion can arise.

\section{The structure of the ideal $\cI_{-n}$}

Under the identifications $$\text{gr}(\cM^+_{-n})\cong \mathbb{C}[Q_{a,b}],\qquad \text{gr}(\cV_{-n})\cong ( \text{Sym} \bigoplus_{k\geq 0} U_k)^{\text{Sp}(2n)}\cong \mathbb{C}[q_{a,b}]/I_n,$$ $\text{gr}(\pi_{-n})$ is just the quotient map \eqref{weylquot}. 
\begin{lemma} \label{lemddef} For each $I = (i_0,\dots, i_{2n+1})$ satisfying \eqref{ijineq}, there exists a unique element \begin{equation} \label{ddef} P_{I}\in (\cM^+_{-n})_{(2n+2)}\cap \cI_{-n}\end{equation} of weight $n+1 +\sum_{a=0}^{2n+1} i_a$, satisfying \begin{equation}\label{uniquedij} \phi_{2n+2}(P_{I}) = p_{I}.\end{equation} These elements generate $\cI_{-n}$ as a vertex algebra ideal. \end{lemma} 
\begin{proof} Clearly $\pi_{-n}$ maps each filtered piece $(\cM^+_{-n})_{(k)}$ onto $(\cV_{-n})_{(k)}$, so the hypotheses of Lemma \ref{idealrecon} are satisfied. Since $I_{n} = \text{Ker} (\text{gr}(\pi_{-n}))$ is generated by the Pfaffians $p_{I}$, we can apply Lemma \ref{idealrecon} to find $P_{I}\in (\cM^+_{-n})_{(2n+2)}\cap \cI_{-n}$ satisfying $\phi_{2n+2}(P_{I}) = p_{I}$, such that $\{P_{I}\}$ generates $\cI_{-n}$. If $P'_{I}$ also satisfies \eqref{uniquedij}, we would have $P_{I} - P'_{I}\in (\cM^+_{-n})_{(2n)} \cap \cI_{-n}$. Since there are no relations in $\cV_{-n}$ of degree less than $2n+2$, we have $P_{I} - P'_{I}=0$. \end{proof}
Let $\bra P_I \ket$ denote the vector space with basis $\{P_I\}$ where $I$ satisfies \eqref{ijineq}. We have $\bra P_I \ket = (\cM^+_{-n})_{(2n+2)}\cap \cI_{-n}$, and clearly $\bra P_I \ket$ is a module over the Lie algebra $\cP^+\subset \hat{\cD}^+$ generated by $\{W^{2m+1}(k) |\ m,k \geq 0\}$, since $\cP^+$ preserves both the filtration on $\cM^+_{-n}$ and the ideal $\cI_{-n}$. It will be convenient to work with the following generating set for $\cP^+$,
$$\{\Omega_{a,b} (a+b-w)|\ 0\leq a< b,\  a+b-w\geq 0\}.$$ Note that $\Omega_{a,b}(a+b-w)$ is homogeneous of weight $w$. The action of $\cP^+$ by derivations of degree zero on $\text{gr}(\cM^+_{-n})$ coming from the vertex Poisson algebra structure is independent of $n$, and is specified by the action on the generators $\Omega_{l,m}$. We compute \begin{equation} \label{actionp} \Omega_{a,b}(a+b-w) (\Omega_{l,m}) = \lambda_{a,b,w,l} (\Omega_{l+w,m}) + \lambda_{a,b,w,m} ( \Omega_{l,m+w}),\end{equation} where
\begin{equation} \label{actionlambdap}\lambda_{a,b,w,l}  =  \bigg\{ \begin{matrix}  (-1)^{b+1} \frac{(b+l)!}{2(l+w-a)!} - (-1)^{a+1} \frac{(a+l)!}{2(l+w-b)!} & l+w-a \geq 0 \cr & \cr 0 & l+w-a <0 \end{matrix}.\end{equation}
The action of $\cP^+$ on $\bra P_I \ket$ is by \lq\lq weighted derivation" in the following sense. Given $I = (i_0,\dots,i_{2n+1})$ and $p= \Omega_{a,b}(a+b-w)\in \cP^+$, we have \begin{equation} \label{weightederivation} p(P_{I}) = \sum_{r=0}^{2n+1} \lambda_r P_{I^r},\end{equation} for lists $I^r = (i_0,\dots, i_{r-1}, i_r + w,i_{r+1},\dots, i_{2n+1})$, and constants $\lambda_r$. If $i_r + w = i_s$ for some $s \in \{1,\dots, 2n+1\}$ we have $\lambda_r = 0$, and otherwise $\lambda_r = (-1)^k \lambda_{a,b,w,i_r}$, where $k$ is the number of transpositions required to transform $I^r$ into an increasing list, as in \eqref{ijineq}.

For each $n\geq 1$, there is a distinguished element $P_0\in \bra P_I \ket$, defined by $$P_0 = P_{I},\qquad I = (0,1,\dots,2n+1).$$ It is the unique element of $\cI_{-n}$ of minimal weight $2(n+1)^2$, and hence is a singular vector in $\cM^+_{-n}$. 

\begin{thm} \label{uniquesv} $P_0$ generates $\mathcal{I}_{-n}$ as a vertex algebra ideal.\end{thm} 

We need a preliminary lemma in order to prove this statement. For simplicity of notation, we take $n=1$, but our lemma holds for any $n$. In this case, $\cS= \cS(n)$ is generated by $\beta, \gamma$. Recall that $\beta_j$, $\gamma_j$ denote the images of $\partial^j\beta$, $\partial^j\gamma$ in $\text{gr}(\cS)$, respectively. Let $W\subset \text{gr}(\cS)$ be the vector space with basis $\{\beta_j, \gamma_j|\ j\geq 0\}$, and for each $m\geq 0$, let $W_m$ be the subspace with basis $\{\beta_j, \gamma_j |\ 0\leq j\leq m\}$. Let $\phi:W\ra W$ be a linear map of weight $w\geq 1$, such that \begin{equation}\label{arbmap} \phi(\beta_j) = c_j \beta_{j+w},\qquad \phi(\gamma_j) = c_j \gamma_{j+w},\qquad c_j \in \mathbb{C}.\end{equation} For example, the restriction of $w^{2k+1}(2k+1-w)$ to $W$ is such a map for $2k+1-w\geq 0$.

\begin{lemma} \label{third} Fix $w\geq 1$ and $m\geq 0$, and let $\phi$ be a linear map satisfying (\ref{arbmap}). Then the restriction $\phi \big|_{W_m}$ can be expressed uniquely as a linear combination of the operators $w^{2k+1}(2k+1-w)\big|_{W_m}$ for $0\leq 2k+1-w \leq 2m+1$.\end{lemma}

\begin{proof} Suppose first that $w$ is even, and let $k_j = j+ \frac{w}{2}$, for $j=0,\dots,m$. In this notation, we need to show that $\phi \big|_{W_m}$ can be expressed uniquely as a linear combination of the operators $w^{2k_j+1}(2j+1)\big|_{W_m}$ for $j=0,\dots,m$. We calculate 
\begin{equation}\label{actionofp} w^{2k_j+1}(2j+1)(\beta_i) = \lambda_{0,2k_j+1,w,i} (\beta_{i+w}),\qquad w^{2k_j+1}(2j+1)(\gamma_i) = \lambda_{0,2k_j+1,w,i} (\gamma_{i+w}),\end{equation} where $\lambda_{0,2k_j+1,w,i}$ is given by \eqref{actionlambdap}. Let $M^w$ be the $(m+1)\times(m+1)$ matrix with entries $M^w_{i,j} =  \lambda_{0,2k_j+1,w,i}$, for $i,j = 0,\dots,m$. Let ${\bf c}\in \mathbb{C}^{m+1}$ be the column vector with transpose $(c_0,\dots,c_m)$. Given an arbitrary linear combination $$\psi = t_0 w^{2k_0+1}(1) + t_1 w^{2k_1+1}(3) + \cdots + t_{m} w^{2k_m+1}(2m+1)$$ of the operators $w^{2k_j+1}(2j+1)$ for $0\leq j\leq m$, let ${\bf t}$ be the column vector with transpose $(t_0,\dots, t_{m})$. Note that $\phi \big|_{W_m} = \psi \big|_{W_m}$ precisely when $M^w  {\bf t} =  {\bf c}$, so in order to prove the claim, it suffices to show that $M^w$ is invertible. But this is clear since each $2\times 2$ minor $$\bigg[ \begin{matrix} M^w_{i,j} & M^w_{i,j+1} \cr  M^w_{i+1,j} & M^w_{i+1,j+1}\end{matrix} \bigg]$$ has positive determinant. If $w$ is odd, the same argument shows that for $k_j = j+\frac{1}{2}(w-1)$, $\phi$ can be expressed uniquely as a linear combination of $w^{2k_j+1}(2j)$ for $j=0,\dots,m$. \end{proof}

Since \eqref{actionofp} holds for all $n\geq 1$ with $\beta_j$ and $\gamma_j$ replaced with $\beta^i_j$ and $\gamma^i_j$, the statement of Lemma \ref{third} holds for any $n$. More precisely, let $W\subset \text{gr}(\cS(n))$ be the vector space with basis $\{\beta^{i}_j, \gamma^{i}_j |i=1,\dots,n,\ j\geq 0\}$, and let $W_m\subset W$ be the subspace with basis $\{\beta^{i}_j, \gamma^{i}_j|i=1,\dots, n,\ 0\leq j\leq m\}$. Let $\phi:W\ra W$ be a linear map of weight $w\geq 1$ taking \begin{equation}\label{actiongencase} \beta^{i}_j\mapsto c_j \beta^{i}_{j+w},\qquad \gamma^{i}_j\mapsto c_j \gamma^{i}_{j+w}, \qquad i=1,\dots,n,\end{equation} where $c_j$ is independent of $i$. Then $\phi\big|_{W_m}$ can be expressed uniquely as a linear combination of $w^{2k+1}(2k+1-w)\big|_{W_m}$ for $0\leq 2k+1-w \leq 2m+1$.

\begin{proof}[Proof of Theorem \ref{uniquesv}]
Since $\mathcal{I}_{-n}$ is generated by $\bra P_I \ket$ as a vertex algebra ideal, it suffices to show that $\bra P_I \ket$ is generated by $P_0$ as a module over $\mathcal{P}^+$. Let $\cI'_{-n}$ denote the ideal in $\cM^+_{-n}$ generated by $P_0$, and let $\bra P_I \ket[m]$ denote the homogeneous subspace of $\bra P_I \ket$ of weight $m$. Note that this space is trivial for $m< 2(n+1)^2$, and is spanned by $P_0$ for $m = 2(n+1)^2$. We will prove that $\bra P_I \ket[m] \subset \cI'_{-n}$ by induction on $m$.

Fix $m > 2(n+1)^2$, and assume that $\bra P_I \ket[m-1] \subset \cI'_{-n}$. Fix $I = (i_0,\dots, i_{2n+1})$ such that $P_I$ has weight $m = n+1+\sum_{k=0}^{2n+1} i_k$. Since $m > 2(n+1)^2$, there is some $k$ for which $i_k >k$. Let $k$ be the first element where this happens, and let $$I' = (i_0,\dots, i_{k-1}, i_k -1, i_{k+1},\dots, i_{2n+1}).$$ Clearly $P_{I'} \in \cI'_{-n}$, and since $i_{k-1} = k-1$, we have $i_k -1 > i_{k-1}$, so $P_{I'} \neq 0$. By Lemma \ref{third}, we can find $p \in \cP^+$ such that $$p(\beta^{i}_{r}) = c_r \beta^{i}_{r+1},\qquad p(\gamma^{i}_{r}) = c_r \gamma^{i}_{r+1}$$ where $c_r = 1$ for $r = i_k -1$ and $c_r=0$ for all other $r\leq i_{2n+1}$. The weighted derivation property \eqref{weightederivation} implies that $p(P_{I'}) = P_I$, so $P_I \in \cI'_{-n}$. \end{proof}

\section{Normal ordering and quantum corrections}

Let $p\in \text{gr}(\cM^+_{-n})\cong \mathbb{C}[Q_{a,b}]$ be homogeneous of degree $k$ in the variables $Q_{a,b}$. By a {\it normal ordering} of $p$, we mean a choice of normally ordered polynomial $P\in (\cM^+_{-n})_{(2k)}$, obtained by replacing $Q_{a,b}$ with $\Omega_{a,b}$, and by replacing ordinary products with iterated Wick products. Of course $P$ is not unique, but for any choice we have $\phi_{2k}(P) = p$. For the rest of this section, $P^{2k}$, $E^{2k}$, $F^{2k}$, etc., will denote elements of $(\cM^+_{-n})_{(2k)}$ which are homogeneous, normally ordered polynomials of degree $k$ in the elements $\Omega_{a,b}$.

Let $P_{I}^{2n+2}\in (\cM^+_{-n})_{(2n+2)}$ be some normal ordering of $p_I$, so that $\phi_{2n+2}(P_I^{2n+2}) = p_I$. Then $$\pi_{-n}(P_I^{2n+2}) \in (\cV_{-n})_{(2n)},$$ and $\phi_{2n}(\pi_{-n}(P_I^{2n+2})) \in \text{gr}(\cV_{-n})$ can be expressed uniquely as a polynomial of degree $n$ in the variables $q_{a,b}$. Choose a normal ordering of the corresponding polynomial in the variables $\Omega_{a,b}$, and call this element $-P^{2n}_{I}$. Then $P^{2n+2}_{I} + P^{2n}_{I}$ satisfies $$\phi_{2n+2}(P_I^{2n+2} + P^{2n}_{I}) = p_I,\qquad \pi_{-n}(P^{2n+2}_{I} + P^{2n}_{I})\in (\cV_{-n})_{(2n-2)}.$$ Continuing this process, we arrive at an element $\sum_{k=1}^{n+1} P^{2k}_{I}$ in the kernel of $\pi_{-n}$, such that $\phi_{2n+2}(\sum_{k=1}^{n+1} P^{2k}_{I}) = p_I$. By Lemma \ref{lemddef}, 
\begin{equation}\label{decompofd} P_{I} = \sum_{k=1}^{n+1}P^{2k}_{I}.\end{equation} The term $P^2_{I}$ lies in the space $A_m$ spanned by $\{\Omega_{a,b}|\ a+b=m\}$, for $m = n+ \sum_{a=0}^{2n+1} i_a$. By \eqref{decompofa}, for all odd $m\geq 1$ we have a projection $$\text{pr}_{m}: A_{m}\ra \bra W^{m}\ket.$$ For all $I = (i_0,\dots, i_{2n+1})$ such that $m = n+ \sum_{a=0}^{2n+1} i_a$ is odd, define the {\it remainder} \begin{equation}\label{defofrij} R_{I} = \text{pr}_m(P^2_{I}).\end{equation}

\begin{lemma} \label{uniquenessofr} Fix $P_{I}\in\cI_{-n}$ with $I = (i_0, \dots, i_{2n+1})$ and $m = n+ \sum_{a=0}^{2n+1} i_a$ odd, as above. Suppose that $P_{I} = \sum_{k=1}^{n+1} P^{2k}_{I}$ and $P_{I} = \sum_{k=1}^{n+1} \tilde{P}^{2k}_{I}$ are two different decompositions of $P_{I}$ of the form (\ref{decompofd}). Then $$P^2_{I} - \tilde{P}^2_{I} \in \partial^2 (A_{m-2}).$$ In particular, $R_{I}$ is independent of the choice of decomposition of $P_{I}$.\end{lemma}

\begin{proof} Since $\cM^+_{-n}$ is a vertex subalgebra of $\cM_{-n}$ and the generators of $\cM^+_{-n}$ are linear combinations of the generators of $\cM_{-n}$ and their derivatives, this follows from Lemma 4.7 of \cite{LI}. \end{proof}

\begin{lemma}\label{lemnonzero} Let $R_0$ denote the remainder of the element $P_0$. The condition $R_0 \neq 0$ is equivalent to the existence of a decoupling relation in $\cV_{-n}$ of the form \begin{equation}\label{maindecoupling} w^{2n^2+4n+1} = Q(w^1,w^3,\dots, w^{2n^2+4n-1}),\end{equation} where $Q$ is a normally ordered polynomial in $w^1, w^3, \dots, w^{2n^2+4n-1}$, and their derivatives. \end{lemma}

\begin{proof} Let $P_0 =  \sum_{k=1}^{n+1}P^{2k}_0$ be a decomposition of $P_0$ of the form (\ref{decompofd}). If $R_0\neq 0$, we have $P^2_0 = \lambda W^{2n^2+4n+1} + \partial^2 \omega$ for some $\lambda \neq 0$ and $\omega\in A_{2n^2+4n-1}$. Since $P_0$ has weight $2(n+1)^2$ and $P^{2k}_0$ has degree $k$ in the elements $W^{2m+1}$ and their derivatives, $P^{2k}_0$ can depend only on $W^1, W^3,\dots, W^{2n^2+4n-1}$, and their derivatives, for $2 \leq k \leq n+1$. Therefore $\frac{1}{\lambda} P_0$ has the form \begin{equation} \label{specialelti} W^{2n^2+4n+1} - Q(W^1, W^3, \dots, W^{2n^2+4n-1}).\end{equation} 
Applying $\pi_{-n}:\cM^+_{-n} \ra \cV_{-n}$ yields the desired relation, since $\pi_{-n}(P_0)=0$. The converse holds since $P_0 \in \cI_{-n}$ is the unique element of weight $2(n+1)^2$, up to scalar multiples. \end{proof}

\begin{lemma} \label{lemnonzeroii}Suppose that $R_0 \neq 0$. Then there exist higher decoupling relations \begin{equation}\label{hdrelm} w^{2m+1} = Q_m(w^1, w^3, \dots, w^{2n^2+4n-1})\end{equation} for all $m > n^2+2n$, where $Q_m$ is a normally ordered polynomial in $w^1, w^3, \dots, w^{2n^2+4n-1}$, and their derivatives. \end{lemma}

\begin{proof} It suffices to find elements $W^{2m+1} - Q_m(W^1, W^3, \dots, W^{2n^2+4n-1})\in \cI_{-n}$, so we assume inductively that $Q_{l}$ exists for $n^2+2n \leq l <m$. Choose a decomposition $$Q_{m-1} = \sum_{k=1}^{d} Q^{2k}_{m-1},$$ where $Q_{m-1}^{2k}$ is a normally ordered polynomial of degree $k$ in $W^1, W^3,\dots, W^{2n^2+4n-1}$, and their derivatives. In particular, $$Q^2_{m-1} = \sum_{i=0}^{n^2+2n-1} c_i \partial^{2m-2i-2} W^{2i+1},$$ for constants $c_0,\dots, c_{n^2+2n-1}$. We apply the operator $W^3\circ_1\in \mathcal{P}^+$, which raises the weight by two. By \eqref{raisingoperator}, we have $$W^3\circ_1 W^{2m-1} = -2m W^{2m+1}+ \sum_{k=1}^m \lambda_k \partial^{2k} W^{2m+1-2k},$$ for constants $\lambda_k$. By inductive assumption, for $0<k \leq m-n^2-2n$, we can use the element $$\partial^{2k} \big(W^{2m+1-2k} - Q_{m-k}(W^1, W^3, \dots, W^{2n^2+4n-1})\big) \in \cI_{-n}$$ to express $\partial^{2k} W^{2m+1-2k}$ as a normally ordered polynomial in $W^1, W^3, \dots, W^{2n^2+4n-1}$, and their derivatives, modulo $\cI_{-n}$.

Moreover, $W^3 \circ_1\big(\sum_{k=1}^{d} Q^{2k}_{m-1}\big)$ can be expressed in the form $\sum_{k=1}^{d} E^{2k}$, where each $E^{2k}$ is a normally ordered polynomial in $W^1, W^3, \dots, W^{2n^2+4n+1}$, and their derivatives. If $W^{2n^2+4n+1}$ or its derivatives appear in $E^{2k}$, we can use \eqref{specialelti} to eliminate $W^{2n^2+4n+1}$ and any of its derivatives, modulo $\cI_{-n}$. Hence $$W^3\circ_1\big(\sum_{k=1}^{d} Q^{2k}_{m-1}\big)$$ can be expressed modulo $\cI_{-n}$ in the form $\sum_{k=1}^{d'} F^{2k}$, where $d'\geq d$, and $F^{2k}$ is a normally ordered polynomial in $W^1, W^3, \dots, W^{2n^2+4n-1}$, and their derivatives. It follows that $$-\frac{1}{2m} W^3\circ_1 \big(W^{2m-1} - Q_{m-1}(W^1,W^3,\dots, W^{2n^2+4n-1})\big)$$ can be expressed as an element of $\cI_{-n}$ of the desired form. \end{proof}

\section{A closed formula for $R_I$}
\label{recursionanalysis}
We shall find a closed formula for $R_I$ for any $I = (i_0,\dots, i_{2n+1})$ such that $\text{wt}(P_I) = n+1+ \sum_{a=0}^{2n+1} i_a$ is even, and it will be clear from our formula that $R_0 \neq 0$. We write \begin{equation} \label{remcoeff} R_{I} = R_n(I) W^m,\qquad m = n+ \sum_{a=0}^{2n+1} i_a\end{equation} so that $R_n(I)$ denotes the coefficient of $W^m$ in $\text{pr}_m(P^2_{I})$. For $n=1$ and $I = (i_0, i_1, i_2, i_3)$ the following formula is easy to obtain using the fact that $\text{pr}_m(\Omega_{a,b}) =(-1)^a W^{m}$ for $m=a+b$.

\begin{equation}\label{startrecur} R_1(I) = \frac{1}{4} \bigg(\frac{(-1)^{i_0 + i_2} - (-1)^{i_1 + i_2} - (-1)^{i_0 + i_3}+ (-1)^{ i_1 + i_3}}{1 + i_0 + i_1}\end{equation}

$$+\frac{ - (-1)^{i_0 + i_1}+ (-1)^{ i_1 + i_2} + (-1)^{i_0 + i_3} - (-1)^{ i_2 + i_3}}{1 + i_0 + i_2}$$

$$+\frac{(-1)^{i_0 + i_1} - (-1)^{i_0 + i_2} - (-1)^{ i_1 + i_3} + (-1)^{ i_2 + i_3}}{1 + i_0 + i_3} $$

$$+ \frac{-(-1)^{i_0 + i_2}  + (-1)^{i_0 + i_1} - (-1)^{i_1 + i_3}  + (-1)^{i_2 + i_3}}{1 + i_1 + i_2}$$

$$+\frac{(-1)^{i_1 + i_2} + (-1)^{i_0 + i_3}  - (-1)^{i_2 + i_3} - (-1)^{i_0 + i_1}}{1 + i_1 + i_3}$$

$$ +\frac{-(-1)^{i_1 + i_2} + (-1)^{i_1 + i_3} - (-1)^{i_0 + i_3} + (-1)^{i_0 + i_2}}{1 + i_2 + i_3} \bigg).$$

First, we need a {\it recursive} formula for $R_n(I)$ in terms of the expressions $R_{n-1}(J)$, so we will assume that $R_{n-1}(J)$ has been defined for all $J$. Recall that $\cS(n)$ is a graded algebra with $\mathbb{Z}_{\geq 0}$ grading (\ref{grading}), which specifies a linear isomorphism $$\cS(n)\cong \text{Sym}\bigoplus_{k\geq 0} U_k,\qquad U_k \cong \mathbb{C}^{2n}.$$ Since $\cV_{-n}$ is a graded subalgebra of $\cS(n)$, we obtain an isomorphism of graded vector spaces \begin{equation}\label{linisomor} i_{-n}: \cV_{-n} \rightarrow (\text{Sym}\bigoplus_{k\geq 0} U_k)^{\text{Sp}(2n)}.\end{equation} Let $p\in (\text{Sym}\bigoplus_{k\geq 0} U_k )^{\text{Sp}(2n)}$ be homogeneous of degree $2d$, and let $$f = (i_{-n})^{-1}(p)\in (\cV_{-n})^{(2d)}$$ be the corresponding homogeneous element. Given $F\in (\cM^+_{-n})_{(2d)}$ satisfying $\pi_{-n}(F) = f$, we can write $F = \sum_{k=1}^{d} F^{2k}$, where $F^{2k}$ is a normally ordered polynomial of degree $k$ in the generators $\Omega_{a,b}$.

For $k\geq 0$, let $\tilde{U}_k$ be a copy of the standard representation $\mathbb{C}^{2n+2}$ of $\text{Sp}(2n+2)$, and let $$\tilde{q}_{a,b}\in (\text{Sym} \bigoplus_{k\geq 0} \tilde{U}_k)^{\text{Sp}(2n+2)}\cong \text{gr}(\cS(n+1))^{\text{Sp}(2n+2)} \cong \text{gr}(\cV_{-n-1})$$ be the generator given by (\ref{weylgenerators}). Let $\tilde{p}$ be the polynomial of degree $2d$ obtained from $p$ by replacing each $q_{a,b}$ with $\tilde{q}_{a,b}$, and let $$\tilde{f} = (i_{-n-1})^{-1} (\tilde{p}) \in (\cV_{-n-1})^{(2d)}$$ be the corresponding homogeneous element. Finally, let $\tilde{F}^{2k}\in \cM^+_{-n-1}$ be the element obtained from $F^{2k}$ by replacing each $\Omega_{a,b}$ with the corresponding generator $\tilde{\Omega}_{a,b}\in \cM^+_{-n-1}$, and let $\tilde{F} = \sum_{i=1}^d \tilde{F}^{2k}$. 

\begin{lemma} \label{corhomo} Fix $n\geq 1$, and let $P_{I}$ be an element of $\cI_{-n}$ given by Lemma \ref{lemddef}. There exists a decomposition $P_{I} = \sum_{k=1}^{n+1} P^{2k}_{I}$ of the form \eqref{decompofd} such that the corresponding element $$\tilde{P}_{I} = \sum_{k=1}^{n+1} \tilde{P}^{2k}_{I} \in \cM^+_{-n-1}$$ has the property that $\pi_{-n-1}(\tilde{P}_{I})$ lies in the subspace $(\cV_{-n-1})^{(2n+2)}$ of degree $2n+2$.
\end{lemma}

\begin{proof} The argument is the same as the proof of Corollary 4.14 of \cite{LI}, and is omitted.
\end{proof}

Recall that the Pfaffian $p_I$ has an expansion $$p_I =  \sum_{r=1}^{2n+1} (-1)^{r+1} q_{i_0,i_r} p_{I_r},$$ where $I_r = (i_1,\dots, \widehat{i_r},\dots, i_{2n+1})$ is obtained from $I$ by omitting $i_0$ and $i_r$. Let $P_{I_r} \in \cM^+_{-n+1}$ be the element corresponding to $p_{I_r}$. By Lemma \ref{corhomo}, there exists a decomposition $$P_{I_r} = \sum_{i=1}^n P^{2i}_{I_r}$$ such that the corresponding element $\tilde{P}_{I_r} = \sum_{i=1}^n \tilde{P}^{2i}_{I_r} \in \cM^+_{-n}$ has the property that $\pi_{-n}(\tilde{P}_{I_r})$ lies in the subspace $(\cV_{-n})^{(2n)}$ of degree $2n$. We have \begin{equation}\label{usefuli} \sum_{r=1}^{2n+1} (-1)^{r+1} :\Omega_{i_0, i_r} \tilde{P}_{I_r}:\  = \sum_{r=1}^{2n+1} \sum_{i=1}^n (-1)^{r+1} :\Omega_{i_0,i_r} \tilde{P}^{2i}_{I_r}:.\end{equation} 

The right hand side of \eqref{usefuli} consists of normally ordered monomials of degree at least $2$ in the generators $\Omega_{a,b}$, and hence contributes nothing to $R_n(I)$. Since $\pi_{-n}(\tilde{P}_{I_r})$ is homogeneous of degree $2n$, $\pi_{-n}(:\Omega_{i_0,i_r} \tilde{P}_{I_r}:)$ consists of a piece of degree $2n+2$ and a piece of degree $2n$ coming from all double contractions of $\Omega_{i_0,i_r}$ with terms in $\tilde{P}_{I_r}$, which lower the degree by two. The component of $$\pi_{-n}\bigg(\sum_{r=1}^{2n+1} (-1)^{r+1} :\Omega_{i_0, i_r} \tilde{P}_{I_r}:\bigg)\in \cV_{-n}$$ in degree $2n+2$ must cancel since this sum corresponds to the Pfaffian $p_I$, which is a relation among the variables $q_{a,b}$. The component of $:\Omega_{i_0,i_r}\tilde{P}_{I_r}:$ in degree $2n$ is
\begin{equation} \label{crunch} S_r= \frac{1}{2} \bigg( (-1)^{i_0} \sum_{a} \frac{\tilde{P}_{I_{r,a}}}{i_0+i_{a} +1} + (-1)^{i_r+1} \sum_{a}  \frac{\tilde{P}_{I_{r,a}}}{i_r+i_{a}+1}\bigg)\end{equation}
In this notation, for $a \in \{i_0,\dots, i_{2n+1}\} \setminus \{i_0, i_r\}$, $I_{r,a}$ is obtained from $I_r = (i_1,\dots, \widehat{i_r},\dots, i_{2n+1})$ by replacing $i_a$ with $i_a+i_0+i_r+1$. It follows that \begin{equation} \label{usefulii} \pi_{-n}\bigg(\sum_{r=1}^{2n+1} (-1)^{r+1} :\Omega_{i_0, i_r} \tilde{P}_{I_r}:\bigg)= \pi_{-n}\bigg(\sum_{r=1}^{2n+1} (-1)^{r+1} S_r \bigg).\end{equation} Combining \eqref{usefuli} and \eqref{usefulii}, we can regard $$\sum_{r=1}^{2n+1} \sum_{i=1}^n (-1)^{r+1} :\Omega_{i_0,i_r} \tilde{P}^{2i}_{I_r}: - \sum_{r=1}^{2n+1} (-1)^{r+1} S_r$$ as a decomposition of $P_{I}$ of the form $P_{I} = \sum_{k=1}^{n+1} P^{2k}_{I}$ where the leading term $P^{2n+2}_{I} = \sum_{r=1}^{2n+1} (-1)^{r+1} :\Omega_{i_0, i_r} \tilde{P}^{2n}_{I_r}:$. It follows that $R_n(I)$ is the negative of the sum of the terms $R_{n-1}(J)$ corresponding to each $\tilde{P}_{J}$ appearing in $\sum_{r=1}^{2n+1} (-1)^{r+1} S_r$. We obtain the following recursive formula:

\begin{equation} \label{recursion}R_n(I) = -\frac{1}{2} \sum_{r=1}^{2n+1} (-1)^{r+1} \bigg( (-1)^{i_0} \sum_{a} \frac{R_{n-1}(I_{r,a})}{i_0+i_{a} +1} + (-1)^{i_r+1} \sum_{a}  \frac{R_{n-1}(I_{r,a})}{i_r+i_{a}+1}\bigg) .\end{equation}

To find a closed formula for $R_n(I)$, we begin with the case $n=1$. It follows from \eqref{startrecur} that $R_1(I) = 0$ unless $I = (i_0, i_1, i_2, i_3)$ contains two even and two odd elements. By permuting $i_0, i_1, i_2,i_3$ if necessary, we may assume that $i_k \equiv k\ \text{mod}\ 2$. In this case, \eqref{startrecur} simplifies as follows: 

\begin{equation} \label{closedone} R_1(I)= \frac{ (2 + i_0 + i_1 + i_2 + i_3)(i_0 - i_2) (i_1 - i_3)}{(1 + i_0 + i_1) (1 + i_1 + i_2) (1 + i_0 + i_3) (1 + i_2 + i_3)}.\end{equation}

It is clear by induction on $n$ that $R_n(I) = 0$ unless $I = (i_0,\dots, i_{2n+1})$ is \lq\lq balanced" in the sense that it has $n+1$ even elements and $n+1$ odd elements. From now on, we assume this is the case, and we change our notation slightly. We write $I = (i_0,j_0, i_1,j_1,\dots, i_n, j_n)$, where $i_0,\dots, i_n$ are even and $j_0,\dots, j_n$ are odd. For later use, we record one more obvious symmetry: for $I = (i_0,j_0, i_1, j_1,\dots, i_n, j_n)$ as above, we have \begin{equation} \label{extrasymm} R_n(I) = (-1)^{n+1} R_n(I'),\qquad I' = (j_0,i_0, j_1,i_1,\dots, j_{n}, i_{n}).\end{equation}

Fix $I= (i_0,j_0,i_1,j_1,\dots, i_{n-1}, j_{n-1})$ with each $i_r$ even and $j_r$ odd. For each even integer $x\geq 0$, set $I_x = (i_0,j_0,i_1,j_1,\dots, i_{n-1}, j_{n-1}, x,x+1)$. We regard $R_n(I_x)$ as a rational function of $x$. By applying the recursive formula \eqref{recursion} $k$ times, we can express $R_n(I_x)$ as a linear combination of terms of the form $R_{m}(K)$ where $m=n-k$ and $K = (k_0,k_1,\dots, k_{2m+1})$. Each entry of $K$ is a constant plus a linear combination of entries from $I_x$, and at most two entries of $K$ depend on $x$. Let $V_m$ denote the vector space spanned by elements $R_m(K)$ with these properties. Using \eqref{recursion}, we can express $R_m(K)$ as a linear combination of elements in $V_{m-1}$. A term $R_{m-1}(K_{r,a})$ in this decomposition will be called {\it $x$-active} if either $k_0$, $k_r$, or $k_a$ depends on $x$. Define linear maps \begin{equation} f_m: V_m \rightarrow V_{m-1},\qquad g_m: V_m \rightarrow V_{m-1}\end{equation} as follows: $f_m(R_{m}(K))$ is the sum of the terms which are not $x$-active, and $g_m(R_m(K))$ is the sum of the terms which are $x$-active. Finally, define the {\it constant term map} \begin{equation} \label{conmap} h_m: V_m \rightarrow \mathbb{Q},\qquad h_m (R_m(K)) = \lim_{x\rightarrow \infty} R_m(K). \end{equation}

\begin{lemma} \label{indy} For all $n\geq 2$ and $I = (i_0,j_0, i_1,j_1,\dots, i_{n-1},j_{n-1})$ with $i_0,\dots, i_{n-1}$ even and $j_0,\dots, j_{n-1}$ odd, we have \begin{equation} \label{constantid}h_n(R_n(I_x))= \frac{n}{n + \sum_{k=0}^{n-1} i_k + j_k} R_{n-1}(I).\end{equation} \end{lemma}

\begin{proof} For $n=2$ this is an easy calculation, so we may proceed by induction on $n$. It will be convenient to prove the following auxiliary formula at the same time:
\begin{equation} \label{constantidii} h_{n-1}(g_n(R_{n}(I_x))) = \frac{1}{n + \sum_{k=0}^{n-1} i_k + j_k} R_{n-1}(I).\end{equation} For $n=2$ this can be checked by direct calculation, so we assume both \eqref{constantid} and \eqref{constantidii} for $n-1$. Each term appearing in $f_n(R_{n}(I_x))$ is of the form $$R_{n-1}(K),\qquad K = (k_0,k_1,\dots, k_{2n-3}, x,x+1),\qquad n-1 + \sum_{t=0}^{2n-3} k_t = n + \sum_{k=0}^{n-1} i_k + j_k.$$ By our inductive hypothesis that \eqref{constantid} holds for $n-1$, we have \begin{equation} \label{indyi}h_{n-1}(f_n(R_n(I_x))) = \frac{n-1}{n + \sum_{k=0}^{n-1} i_k + j_k} R_{n-1}(I).\end{equation}
Next, it is easy to check that $$h_{n-2}(g_{n-1}(f_n(R_{n}(I_x)))) = h_{n-2}(f_{n-1}(g_n(R_n(I_x)))).$$
Also, we have $$g_{n-1}(g_n(R_n(I_x))) = 0,$$ since all terms in this expression cancel pairwise. Therefore $$h_{n-1}(g_n(R_n(I_x))) = h_{n-2} (f_{n-1} (g_n(R_n(I_x)))) + h_{n-2} (g_{n-1} (g_n(R_n(I_x)))) $$ $$ = h_{n-2} (f_{n-1} (g_n(R_n(I_x)))) = h_{n-2} (g_{n-1} (f_n(R_n(I_x)))). $$ Moreover, by applying the induction hypothesis that \eqref{constantidii} holds for $n-1$, it follows that \begin{equation} \label{indyii}h_{n-1}(g_n(R_n(I_x))) = h_{n-2} (g_{n-1} (f_n(R_n(I_x)))) = \frac{1}{n + \sum_{k=0}^{n-1} i_k + j_k} R_{n-1}(I).\end{equation} Since $h_n(R_n(I_x)) = h_{n-1}(f_n(R_n(I_x)))+h_{n-1}(g_n(R_n(I_x)))$, the claim follows from \eqref{indyi} and \eqref{indyii}. \end{proof}

\begin{thm} \label{closedform} Suppose that $I = (i_0,j_0, i_1,j_1,\dots, i_n, j_n)$ is a list of nonnegative integers such that  $i_0,\dots, i_n$ are even and $j_0,\dots, j_n$ are odd, as above. Then the following closed formula holds:
\begin{equation} \label{closedeq} R_n(I) = \frac{ n! \bigg(n+1 + \sum_{k=0}^{n} i_k+j_k\bigg)  \bigg(\prod_{0\leq k<l\leq n} (i_k - i_l)(j_k - j_l) \bigg)}{\prod_{0\leq k \leq n, \  0\leq l \leq n} (1+i_k + j_l)}.\end{equation} \end{thm}

\begin{proof} As we shall see, this formula is more or less forced upon us by the symmetries of the Pfaffians, which are inherited by the numbers $R_n(I)$. We assume that \eqref{closedeq} holds for $n-1$ and all $K = (k_0,l_0,k_1,l_1,\dots, k_{n-1}, l_{n-1})$ with $k_i$ even and $l_i$ odd, and we proceed by induction on $n$. In particular, each term of the form $R_{n-1}(K)$ is a rational function in the entries of $K$, where the denominator has degree $n^2$ and the numerator has degree $n^2-n+1$. Hence the total degree of $R_{n-1}(K)$ is $-n+1$.

We may expand $R_n(I)$ as a rational function of $i_0,j_0, i_1,j_1, \dots, i_n,j_n$, using \eqref{recursion}. Each term in \eqref{recursion} has degree $-n$, so $R_n(I)$ has degree at most $-n$. The denominator of each such term is clearly a product of factors of the form
$$1+i_k+j_l,\ \ \ \ 1+i_k + i_l,\ \ \ \ 1+j_k + j_l,\ \ \ \ 2+i_0 + i_k + i_l,\ \ \ \ 2+ i_0 + j_k + i_l,\ \ \ \ 2+i_0+j_k + j_l.$$
As a rational function of $i_0,j_0, i_1,j_1, \dots, i_n,j_n$, the denominator of $R_n(I)$ is therefore the product of a subset of these terms. Moreover, $R_n(I)$ has the following symmetry; for every interchange of $i_k$ and $i_l$, we pick up a sign, and for every interchange of $j_k$ and $j_l$ we pick up a sign. These permutations must have the effect of permuting the factors in the denominator up to a sign, and permuting the factors of the numerator up to a sign. 

From this symmetry, it is clear that none of the expressions $$2+i_0 + i_k + i_l,\qquad 2+ i_0 + j_k + i_l,\qquad 2+i_0+j_k + j_l$$ can appear. Each term $R_{n-1}(I_{r,a})$ appearing in \eqref{recursion} with $i_a$ and $i_r$ both even must vanish, since $I_{r,a}$ will then contain $n-2$ even elements and $n+2$ odd elements. Hence none of the factors $1+i_k+i_l$ can appear in the denominator of $R_n(I)$. By \eqref{extrasymm}, the factors $1+j_k+j_l$ also cannot appear, so the denominator of $R_n(I)$ can contain only a subset of expressions of the form $1+i_k + j_l$. By symmetry, it must contain {\it all} such expressions, so the denominator is precisely $$\prod_{0\leq k \leq n,\ 0\leq l \leq n} (1+ i_k + j_l),$$ and therefore has leading degree $(n+1)^2$. Since $R_n(I)$ picks up a sign under permutation of $i_k$ and $i_l$, and under permutation of $j_k$ and $j_l$, the numerator must be divisible by $$\prod_{0\leq k<l\leq n} (i_k - i_l)( j_k - j_l),$$ which is homogeneous of degree $n^2+n$. Since the denominator has degree $(n+1)^2$, the total degree of $R_n(I)$ is at least $-n-1$. 

We have already seen that the total degree of $R_n(I)$ is at most $-n$, so there is room for at most one more linear factor in the numerator. By symmetry, this factor must be invariant under all permutations of $i_0,\dots, i_n$ and all permutations of $j_0,\dots, j_n$, so it has the form $$a_n \bigg(\sum_{k=0}^n i_k\bigg) + b_n \bigg(\sum_{k=0}^n j_k\bigg)+ c_n,$$ where $a_n, b_n, c_n$ are constants which depend on $n$ but not on $I$. The additional symmetry \eqref{extrasymm} shows that $a_n = b_n$, so this can be rewritten in the form 
$$a_n \bigg(\sum_{k=0}^n i_k+j_k\bigg) + c_n.$$
In order to complete the proof of Theorem \ref{closedform}, it suffices to show that 
\begin{equation} \label{aandb} a_n = n!,\qquad c_n = (n+1)!.\end{equation} 
By our inductive assumption, $a_{n-1} = (n-1)!$. Fix $$K = (k_0,l_0, k_1,l_1,\dots, k_{n-1}, l_{n-1})$$ with each $k_r$ even and $l_r$ odd. As in Lemma \ref{indy}, for each even integer $x\geq 0$, set $$K_x = (k_0,l_0,k_1,l_1,\dots, k_{n-1}, l_{n-1}, x,x+1).$$ The highest power of $x$ appearing in the numerator of $R_n(K_x)$ is $x^{2n+1}$, and the coefficient of $x^{2n+1}$ in the numerator is $$2 a_n \prod_{0\leq r<s<n} (k_r - k_s) (l_r-l_s).$$ Similarly, the highest power of $x$ appearing in the denominator of $R_{n}(K_x)$ is also $x^{2n+1}$, and the coefficient of $x^{2n+1}$ is $$2 \prod_{0\leq r <n,\ 0 \leq s <n} (1+ k_r+ l_s).$$ Therefore $$h_n(R_n (K_x)) = \lim_{x\rightarrow \infty} R_n (K_x) = \frac{a_n \prod_{0\leq r<s <n} (k_r-k_s) (l_r-l_s)}{\prod_{0\leq r <n,\ 0\leq s <n} (1+ k_r+ l_s)}.$$ By our inductive assumption, we have $$R_{n-1}(K) =\frac{(n-1)! (n+ \sum_{r=0}^{n-1} k_r + l_r) \prod_{0\leq r<s <n} (k_r-k_s) (l_r-l_s)}{\prod_{0\leq r <n,\ 0\leq s <n} (1+ k_r+ l_s)}.$$ Therefore by \eqref{constantid}, we have $$h_n(R_n (K_x)) =\bigg( \frac{n}{n+ \sum_{r=0}^{n-1} k_r+l_r} \bigg)\frac{(n-1)! (n+ \sum_{r=0}^{n-1} k_r+l_r) \prod_{0\leq r<s <n} (k_r-k_s) (l_r-l_s)}{\prod_{0\leq r <n,\ 0\leq s <n} (1+ k_r+ l_s)}.$$ This proves that $a_n = n!$. 

Finally, we need to show that $c_n = (n+1)!$. Let $I = (i_0, j_0, i_1,j_1,\dots, i_n, j_n)$ as above. Since $a_n = n!$, it suffices to show that the numerator of $R_n(I)$ is divisible by $n+1+\sum_{k=0}^n i_k + j_k$. But this is clear from \eqref{recursion}, since by inductive assumption, the numerator of each term of the form $R_{n-1}(I_{r,a})$ is divisible by $n+1+\sum_{k=0}^n i_k + j_k$. \end{proof}

\begin{thm} \label{descriptionofvn} For all $n\geq 1$, $\cV_{-n}$ has a minimal strong generating set $\{w^1, w^3,\dots, w^{2n^2+4n-1}\}$, and is therefore of type $\cW(2,4,\dots, 2n^2+4n)$.
\end{thm}

\begin{proof} Specializing \eqref{closedeq} to the case $I = (0,1,\dots,2n+1)$ yields $$R_n(I) = \frac{n! (n+1)^2 \prod_{0\leq k<l \leq n} (k-l)^2}{2^n \prod_{0\leq k \leq n,\ 0\leq l \leq n} (1+k + l)}.$$ In particular, $R_0 = R_n(I) W^{2n^2+4n+1} \neq 0$, so the claim follows from Lemma \ref{lemnonzeroii}. \end{proof}

\section{The case $c=\frac{n}{2}$ for $n\geq 1$}

For each integer $n\geq 1$, $\cV_{n/2}$ admits a free field realization as the $\text{O}(n)$-invariant subalgebra of the free fermion algebra $\cF(n)$ of rank $n$, which has odd generators $\phi^1,\dots, \phi^n$ satisfying
$$\phi^i(z) \phi^j (w) \sim \delta_{i,j} (z-w)^{-1}.$$ We give $\cF(n)$ the conformal structure \begin{equation}\label{bconeparameter} 
L = - \frac{1}{2} \sum_{i=1}^n  :\phi^i \partial \phi^i:
\end{equation} of central charge $\frac{n}{2}$, under which $\phi^i$ is primary of weight $\frac{1}{2}$. The full automorphism group of $\cF(n)$ is $\text{O}(n)$, and $\{\phi^1,\dots, \phi^n\}$ spans a copy of the standard $\text{O}(n)$-module $\mathbb{C}^n$. There is a basis of $\cF(n)$ consisting of normally ordered monomials
\begin{equation}\label{basisofe} :\partial^{I_1} \phi^{1}\cdots \partial^{I_n} \phi^{n}: .\end{equation} In this notation, $I_k = (i^k_1,\dots, i^k_{r_k})$ are lists of integers satisfying $0\leq i^k_1 < \cdots < i^k_{r_k}$, and 
$$\partial^{I_k} \phi^{k} = \ :\partial^{i^k_1} \phi^{k} \cdots \partial^{i^k_{r_k}} \phi^{k}:.$$ We have a $\mathbb{Z}_{\geq 0}$-grading \begin{equation}\label{gradingodd} \cF(n) = \bigoplus_{d\geq 0} \cF(n)^{(d)},\end{equation} where $\cF(n)^{(d)}$ is spanned by monomials of the form \eqref{basisofe} of total degree $ d = \sum_{k=1}^n r_k$. The filtration $\cF(n)_{(d)} = \bigoplus_{i=0}^d \cF(n)^{(i)}$ satisfies \eqref{goodi} and we have an isomorphism of $\text{O}(n)$-modules
$\cF(n) \cong  \text{gr}(\cF(n))$,
and an isomorphism of supercommutative rings
$ \text{gr}(\cF(n))\cong \bigwedge \bigoplus_{k\geq 0} U_k$. Here $U_k$ is the copy of the standard $\text{O}(n)$-module $\mathbb{C}^{n}$ spanned by $\{\phi^{1}_k, \dots \phi^n_k\}$, where $\phi^{i}_k$ is the image of $\partial^k \phi^{i}(z)$ in $\text{gr}(\cF(n))$.

The following result appears as Proposition 14.1 of \cite{KWY} for an even integer $n=2l$, and as Proposition 14.2 for $n=2l+1$. Note that $\cF(2l)$ and $\cF(2l+1)$ are isomorphic to the vertex algebras $\cF^{\otimes l}$ and $\cF^{\otimes l+\frac{1}{2}}$ in \cite{KWY}.

\begin{thm} There is an isomorphism $\cV_{n/2}\rightarrow \cF(n)^{\text{O}(n)}$ given by
\begin{equation} \label{bcrealization} w^{2m+1} \mapsto - \frac{1}{2} \sum_{i=1}^n :\phi^i \partial^{2m+1} \phi^i: ,\qquad m\geq 0.\end{equation} 
\end{thm}
Since $\text{O}(n)$ preserves the grading on $\cF(n)$, $\cV_{n/2} \subset \cF(n)$ is a graded subalgebra. Write \begin{equation}\label{gradingvodd} \cV_{n/2} = \bigoplus_{d\geq 0} ( \cV_{n/2})^{(d)},\qquad (\cV_{n/2})^{(d)} = \cV_{n/2}\cap\cF(n)^{(d)},\end{equation} and define the corresponding filtration by $(\cV_{n/2} )_{(d)} = \bigoplus_{i=0}^{d} (\cV_{n/2} )^{(i)}$. Then \eqref{bcrealization} preserves conformal structures and is a morphism in the category $\cR$. We clearly have \begin{equation} \label{grisosodd} \text{gr}(\cV_{n/2}) \cong \text{gr}(\cF(n)^{\text{O}(n)}) \cong \text{gr}(\cF(n))^{\text{O}(n)} \cong (\bigwedge \bigoplus_{k\geq 0} U_k)^{\text{O}(n)}.\end{equation} We recall the following analogue of Weyl's first and second fundamental theorems of invariant theory for the orthogonal group (Theorems 2.9A and 2.17A of \cite{We}), where the symmetric algebra is replaced by the exterior algebra. In fact, this theorem is a special case of Sergeev's first and second fundamental theorems of invariant theory for $\text{Osp}(m,2n)$ (Theorem 1.3 of \cite{SI} and Theorem 4.5 of \cite{SII}).

\begin{thm} \label{weylfftorthodd} For $k\geq 0$, let $U_k$ be the copy of the standard $\text{O}(n)$-module $\mathbb{C}^{n}$ with orthonormal basis $\{x_{i,k} |\ i=1,\dots,n\}$. Then $(\bigwedge \bigoplus_{k\geq 0} U_k )^{\text{O}(n)}$ is generated by the quadratics \begin{equation}\label{weylgeneratorsodd} q_{a,b} = \frac{1}{2}\sum_{i=1}^n  x_{i,a} x_{i,b} ,\qquad a,b \geq 0. \end{equation} Note that $q_{a,b} = -q_{b,a}$. Let $\{Q_{a,b}|\ a,b\geq 0\}$ be commuting indeterminates satisfying $Q_{a,b} = -Q_{b,a}$. The kernel $I_n$ of the homomorphism \begin{equation}\label{weylquotodd} \mathbb{C}[Q_{a,b}]\ra (\bigwedge \bigoplus_{k\geq 0} U_k)^{\text{O}(n)},\qquad Q_{a,b}\mapsto q_{a,b},\end{equation} is generated by polynomials $d_{I,J}$ degree $n+1$, which are indexed by lists $I = (i_0,\dots,i_{n})$ and $J = (j_0, \dots, j_{n})$ of integers satisfying \begin{equation}\label{ijineqodd} 0\leq i_0\leq \cdots \leq i_{n},\qquad 0\leq j_0\leq \cdots \leq j_{n}.\end{equation} These relations are analogous to $(n+1)\times(n+1)$ determinants, but without the usual signs. For $n=1$, $I = (i_0, i_1)$, $J = (j_0, j_1)$, $d_{I,J}$ is given by
$$q_{i_0, j_0} q_{i_1, j_1}  + q_{i_1, j_0} q_{i_0, j_1} $$ and for $n>1$, $d_{I,J}$ is defined inductively by \begin{equation} \label{determinantinduction}d_{I,J}=  \sum_{r=0}^{n} q_{i_r,j_0} d_{I_r, J'},\end{equation} where $I_r = (i_0,\dots, \widehat{i_r},\dots, i_{n})$ is obtained from $I$ by omitting $i_r$, and $J' = (j_1,\dots, j_{n})$ is obtained from $J$ by eliminating $j_0$. \end{thm}

Under \eqref{grisosodd}, the generators $q_{a,b}$ correspond to strong generators
\begin{equation}\label{newgenomegaodd} \omega_{a,b} = \frac{1}{2}\sum_{i=1}^n :\partial^a \phi^i \partial ^b \phi^i:\end{equation} of $\cV_{n/2}$. In this notation, $w^{2m+1} = - \omega_{0,2m+1}$. Letting $A_m$ denote the vector space spanned by $\{\omega_{a,b}|\ a+b = m\}$, we have \begin{equation}\label{decompofaodd} A_{2m+1} = \partial (A_{2m})\oplus \bra w^{2m+1}\ket =  \partial^2 (A_{2m-1})\oplus \bra w^{2m+1}\ket ,\end{equation} where $\bra w^{2m+1}\ket$ is the linear span of $w^{2m+1}$. We have a similar strong generating set $\{\Omega_{a,b}\}$ for $\cM^+_{n/2}$ satisfying $\pi_{n/2}(\Omega_{a,b}) = \omega_{a,b}$. Under the identifications $$\text{gr}(\cM^+_{n/2})\cong \mathbb{C}[Q_{a,b}],\qquad \text{gr}(\cV_{n/2})\cong ( \bigwedge \bigoplus_{k\geq 0} U_k)^{\text{O}(n)}\cong \mathbb{C}[q_{a,b}]/I_n,$$ $\text{gr}(\pi_{n/2})$ coincides with \eqref{weylquotodd}. 

Note that $d_{I,J}$ is nontrivial if and only if no integer $k$ appears more than $n+1$ times in the set $I\cup J = \{ i_0,\dots, i_{n},j_0,\dots, j_{n}\}$, since otherwise $d_{I,J}$ is divisible by $Q_{k,k}$. We have the following analogue of Lemma \ref{lemddef}.
\begin{lemma} \label{lemddefodd} For each pair $I = (i_0,\dots, i_{n})$ and $J = (j_0,\dots, j_{n})$ satisfying \eqref{ijineqodd} such that $d_{I,J}$ is nontrivial, there exists a unique element \begin{equation} \label{ddefodd} D_{I,J}\in (\cM^+_{n/2})_{(2n+2)}\cap \cI_{n/2}\end{equation} of weight $n+1 +\sum_{a=0}^{n} i_a+j_a$, satisfying \begin{equation}\label{uniquedijodd} \phi_{2n+2}(D_{I,J}) = d_{I,J}.\end{equation} These elements generate $\cI_{n/2}$ as a vertex algebra ideal.\end{lemma} 

The relation of minimal weight $2n+2$ occurs when $I = (0,\dots, 0)$ and $J = (1,\dots,1)$, and we denote it by $D_0$. We have the following analogue of Lemma \ref{third}.

\begin{lemma} \label{thirdodd} Let $W\subset \text{gr}(\cF(n))$ be the vector space with basis $$\{\phi^i_j |i=1,\dots,n,\ j\geq 0\},$$ and let $W_m\subset W$ be the subspace with basis $\{\phi^i_j|i=1,\dots, n,\ 0\leq j\leq m\}$. Let $\phi:W\ra W$ be a linear map of weight $w\geq 1$ taking \begin{equation}\label{actiongencase} \phi^i_j\mapsto \lambda_j \phi^i_{j+w}, \qquad i=1,\dots,n,\end{equation} where $\lambda_j$ is independent of $i$. Then $\phi\big|_{W_m}$ can be expressed uniquely as a linear combination of $w^{2k+1}(2k+1-w)\big|_{W_m}$ for $0\leq 2k+1-w \leq 2m+1$. \end{lemma}

\begin{thm} \label{uniquesvodd} The relation $D_0$ generates $\cI_{n/2}$ as a vertex algebra ideal.
\end{thm}

\begin{proof}
The argument is similar to the proof of Theorem \ref{uniquesv} but slightly more complicated, so the details are included. Let $\bra D_{I,J} \ket$ denote the vector space spanned by $\{D_{I,J}\}$ where $I$ and $J$ satisfy \eqref{ijineqodd}. We have $\bra D_{I,J} \ket = (\cM^+_{n/2})_{(2n+2)}\cap \cI_{n/2}$, and clearly $\bra D_{I,J} \ket$ is a module over the Lie algebra $\cP^+\subset \hat{\cD}^+$ generated by $\{W^{2m+1}(k) |\ m,k \geq 0\}$. Since $\mathcal{I}_{n/2}$ is generated by $\bra D_{I,J} \ket$ as a vertex algebra ideal, it suffices to show that $\bra D_{I ,J} \ket$ is generated by $D_0$ as a module over $\mathcal{P}^+$. Let $\cI'_{n/2}$ denote the ideal in $\cM^+_{n/2}$ generated by $D_0$, and let $\bra D_{I,J} \ket[m]$ denote the homogeneous subspace of $\bra D_{I,J} \ket$ of weight $m$. This space is trivial for $m< 2n+2$, and is spanned by $D_0$ for $m = 2n+2$. We will prove that $\bra D_{I,J} \ket[m] \subset \cI'_{n/2}$ by induction on $m$.

Fix $D_{I,J}$ of weight $m>2n+2$, and let $k$ be the minimal entry appearing in either $I$ or $J$. We may assume that $k$ appears $r$ times in $I \cup J$, with $r\leq n+1$. Suppose first that $k$ only appears in $I$, so that $J = (j_0,\dots, j_{n})$ with $j_0>k$. If $r<n+1$, we have $I = (k,\dots, k,i_{r},\dots, i_{n})$ with $i_r >k$. In this case, we can choose $p \in \cP^+$ by Lemma \ref{thirdodd} so that $$p(\phi^{i}_{k}) =  \phi^{i}_{i_{r}}, \qquad p(\phi^{i}_t) = 0,\qquad t\neq k,\qquad i=1,\dots, n,$$ where $t$ is less than or equal to the maximal entry appearing in either $I$ or $J$. Set $I' = (k,\dots, k,k,i_{r+1},\dots, i_{n})$, which is obtained from $I$ by replacing $i_r$ with $k$. Since $i_r >k$, $\text{wt}(D_{I',J}) < m$, so by inductive assumption $D_{I',J}$ lies in $\cI'_{n/2}$. By a weighted derivation property analogous to \eqref{weightederivation}, $\frac{1}{r+1} (p(D_{I',J})) = D_{I,J}$, so $D_{I,J}$ lies in $\cI'_{n/2}$ as well.

Next, suppose that $k$ only appears in $I$, and $r=n+1$, so that $I = (k,\dots, k)$. If $k>0$, let $I' = (k-1,k,\dots,k)$. Choose $p\in \cP^+$ such that $$p(\phi^{i}_{k-1}) =  \phi^{i}_{k}, \qquad p(\phi^{i}_t) = 0,\qquad t\neq k-1,\qquad i=1,\dots, n.$$ Then $D_{I',J} \in \cI'_{n/2}$ and $p(D_{I',J}) =D_{I,J}$, so $D_{I,J} \in \cI'_{n/2}$.

Next, suppose that $k$ only appears in $I$, $r=n+1$, and $k=0$, so that $I = (0,\dots, 0)$. Let $l$ be the minimal entry appearing in $J$, and suppose that $l$ appears $s$ times in $J$. If $s<n+1$, we have $J = (l,\dots, l, j_s,\dots, j_{n})$, with $j_s >l$. In this case, let $J' = (l,\dots, l,l,j_{s+1},\dots, j_{n})$ where $j_s$ has been replaced with $l$, and choose $p\in \cP^+$ such that $$p(\phi^{i}_{l}) =  \phi^{i}_{j_{s}}, \qquad p(\phi^{i}_t) = 0,\qquad t\neq l,\qquad i=1,\dots, n.$$ Then $D_{I,J'} \in \cI'_{n/2}$ and $\frac{1}{s+1}(p(D_{I,J'})) = D_{I,J}$, so $D_{I,J} \in \cI'_{n/2}$. 

Next suppose that $I = (0,\dots, 0)$ and $s = n+1$, so that $J = (l,\dots, l)$. Since $m > 2n+2$ we have $l\geq 2$. Take $J' = (l-1,l,\dots,l)$ and choose $p\in \cP^+$ such that $$p(\phi^{i}_{l-1}) =  \phi^{i}_{l}, \qquad p(\phi^{i}_t) = 0,\qquad t\neq l-1,\qquad i=1,\dots, n.$$ Then $D_{I,J'} \in \cI'_{n/2}$ and $p(D_{I,J'}) = D_{I,J}$, so $D_{I,J} \in \cI'_{n/2}$.

Finally, suppose that not all the $k$'s appear in $I$, so that $I = (k,\dots, k, i_a,\dots, i_{n})$ and $J = (k,\dots,k, j_b,\dots, j_{n})$, with $a+b= r$ and $i_a, j_b >k$. We have already established the result for $b=0$, so we proceed by induction on $b$. Let $I' = (k,\dots, k,k,i_{a+1},\dots, i_{n})$, which is obtained from $I$ by replacing $i_a$ with $k$. Choose $p\in \cP^+$ such that $$p(\phi^{i}_{k}) =  \phi^{i}_{i_a}, \qquad p(\phi^{i}_t) = 0,\qquad t\neq k,\qquad i=1,\dots, n.$$
Then $D_{I',J} \in \cI'_{n/2}$ and $p(D_{I',J}) = D_{I,J} + D_{I',J'}$ where $J' = (k,\dots,k,i_a,j_b,\dots, j_{n})$ is obtained from $J$ by replacing the last $k$ with $i_a$. Since $k$ appears $a+1$ times in $I'$, and $b-1$ times in $J'$, we have $D_{I',J'} \in \cI'_{n/2}$ by inductive assumption. It follows that $D_{I,J} \in \cI'_{n/2}$, as desired. \end{proof}

Each $D_{I,J}$ can be written in the form 
\begin{equation}\label{decompofdodd} D_{I,J} = \sum_{k=1}^{n+1}D^{2k}_{I,J}.\end{equation} The term $D^2_{I,J}$ lies in the space $A_m$ spanned by $\{\Omega_{a,b}|\ a+b=m\}$, for $m = n+ \sum_{a=0}^{n} i_a + j_a$. By \eqref{decompofaodd}, for all odd integers $m\geq 1$ we have the projection $$\text{pr}_{m}: A_{m}\ra \bra W^{m}\ket.$$ For all $I = (i_0,\dots, i_{n})$ and $J = (j_0,\dots, j_{n})$ such that $m = n+ \sum_{a=0}^{n} i_a+j_a$ is odd, define the remainder \begin{equation}\label{defofrijodd} R_{I,J} = \text{pr}_m(D^2_{I,J}).\end{equation}
It is independent of the choice of decomposition \eqref{decompofdodd}. In the case $I = (0,\dots,0) $ and $J = (1,\dots, 1)$, we denote $R_{I,J}$ by $R_0$. The condition $R_0 \neq 0$ is equivalent to the existence of a decoupling relation in $\cV_{n/2}$ of the form
 \begin{equation}\label{maindecouplingodd} w^{2n+1} = Q(w^1,w^3,\dots, w^{2n-1}),\end{equation} where $Q$ is a normally ordered polynomial in $w^1, w^3,  \dots, w^{2n-1}$, and their derivatives. We have the following analogue of Lemma \ref{lemnonzeroii}.
 \begin{lemma} \label{lemnonzerodd} If $R_0 \neq 0$, there exist higher decoupling relations \begin{equation}\label{hdrelmodd} w^{2m+1} = Q_m(w^1, w^3, \dots, w^{2n-1})\end{equation} for all $m>n$, where $Q_m$ is a normally ordered polynomial in $w^1, w^3, \dots, w^{2n-1}$, and their derivatives. \end{lemma}

\section{A recursive formula for $R_{I,J}$}
\label{recursionanalysisodd}
For any $I = (i_0,\dots, i_{n})$ and $J = (j_0,\dots, j_{n})$ such that $\text{wt}(D_{I,J}) = n+1+ \sum_{a=0}^{n} i_a+j_a$ is even, we shall find a recursive formula for $R_{I,J}$ in terms of $R_{n-1}(K,L)$ for lists $K,L$ of length $n$. We write \begin{equation} \label{remcoeffodd} R_{I,J} = R_n(I,J) W^m,\qquad m = n + \sum_{a=0}^{n} i_a+j_a\end{equation} so that $R_n(I,J)$ denotes the coefficient of $W^m$ in $\text{pr}_m(D^2_{I,J})$. 

Recall that $\cF(n)$ has a $\mathbb{Z}_{\geq 0}$ grading (\ref{gradingodd}), which specifies a linear isomorphism $$\cF(n)\cong \bigwedge \bigoplus_{k\geq 0} U_k,\qquad U_k \cong \mathbb{C}^{n}.$$ Since $\cV_{n/2}$ is a graded subalgebra of $\cF(n)$, we obtain an isomorphism of graded vector spaces \begin{equation}\label{linisomorodd} i_{n/2}: \cV_{n/2} \rightarrow (\bigwedge \bigoplus_{k\geq 0} U_k)^{\text{O}(n)}.\end{equation} Let $p\in (\bigwedge \bigoplus_{k\geq 0} U_k )^{\text{O}(n)}$ be homogeneous of degree $2d$, and let $$f = (i_{n/2})^{-1}(p)\in (\cV_{n/2})^{(2d)}$$ be the corresponding homogeneous element. Given $F\in (\cM^+_{n/2})_{(2d)}$ satisfying $\pi_{n/2}(F) = f$, we can write $F = \sum_{k=1}^{d} F^{2k}$, where $F^{2k}$ is a normally ordered polynomial of degree $k$ in the generators $\Omega_{a,b}$.

For $k\geq 0$, let $\tilde{U}_k$ be a copy of the standard representation $\mathbb{C}^{n+1}$ of $\text{O}(n+1)$, and let $$\tilde{q}_{a,b}\in (\bigwedge \bigoplus_{k\geq 0} \tilde{U}_k)^{\text{O}(n+1)}\cong \text{gr}(\cF(n+1)^{\text{O}(n+1)}) \cong \text{gr}(\cV_{(n+1)/2})$$ be the generator given by \eqref{weylgeneratorsodd}. Let $\tilde{p}$ be the element of degree $2d$ obtained from $p$ by replacing each $q_{a,b}$ with $\tilde{q}_{a,b}$, and let $$\tilde{f} = (i_{(n+1)/2})^{-1} (\tilde{p}) \in (\cV_{(n+1)/2})^{(2d)}$$ be the corresponding homogeneous element. Finally, let $\tilde{F}^{2k}\in \cM^+_{(n+1)/2}$ be the element obtained from $F^{2k}$ by replacing each $\Omega_{a,b}$ with the corresponding generator $\tilde{\Omega}_{a,b}\in \cM^+_{(n+1)/2}$, and let $\tilde{F} = \sum_{i=1}^d \tilde{F}^{2k}$. We have the following analogue of Lemma \ref{corhomo}.

\begin{lemma} \label{corhomodd} Fix $n\geq 1$, and let $D_{I,J}$ be the element of $\cI_{n/2}$ corresponding to $d_{I,J}$. There exists a decomposition $D_{I,J} = \sum_{k=1}^{n+1} D^{2k}_{I,J}$ of the form \eqref{decompofdodd} such that the corresponding element $$\tilde{D}_{I,J} = \sum_{k=1}^{n+1} \tilde{D}^{2k}_{I,J} \in \cM^+_{(n+1)/2}$$ has the property that $\pi_{(n+1)/2}(\tilde{D}_{I,J})$ lies in the subspace $(\cV_{(n+1)/2})^{(2n+2)}$ of degree $2n+2$.
\end{lemma}

Now we are ready to express $R_n(I,J)$ for all $I=(i_0,\dots, i_{n})$, $J= (j_0,\dots, j_{n})$ in terms of $R_{n-1}(K,L)$ for lists $K,L$ of length $n$. Let $d_{I,J}$ and $D_{I,J}$ be the corresponding elements of $\mathbb{C}[Q_{j,k}]$ and $\cM^+_{n/2}$, respectively. Recall that $d_{I,J} = \sum_{r=0}^{n} Q_{i_r,j_0} d_{I_r, J'}$, where $I_r = (i_0,\dots, \widehat{i_r},\dots, i_{n})$ is obtained from $I$ by omitting $i_r$, and $J' = (j_1,\dots,j_{n})$ is obtained from $J$ by omitting $j_0$. 

Let $D_{I_r, J'} \in \cM^+_{(n-1)/2}$ be the element corresponding to $d_{I_r,J'}$. By Lemma \ref{corhomodd}, there exists a decomposition $$D_{I_r, J'} = \sum_{i=1}^{n} D^{2i}_{I_r, J'}$$ such that the corresponding element $\tilde{D}_{I_r, J'} = \sum_{i=1}^{2n} \tilde{D}^{2i}_{I_r, J'} \in \cM^+_{n/2}$ has the property that $\pi_{n/2}(\tilde{D}_{I_r, J'})$ lies in the subspace $(\cV_{n/2})^{(2n)}$ of degree $2n$. We have \begin{equation}\label{usefuliodd} \sum_{r=0}^{n} :\Omega_{i_r, j_0} \tilde{D}_{I_r,J'}:\  = \sum_{r=0}^{n} \sum_{i=1}^{n}  :\Omega_{i_r, j_0} \tilde{D}^{2i}_{I_r,J'}:.\end{equation} The right hand side of \eqref{usefuliodd} consists of normally ordered monomials of degree at least $2$ in the generators $\Omega_{a,b}$, and hence contributes nothing to $R_{I,J}$. Since $\pi_{n/2}(\tilde{D}_{I_r,J'})$ is homogeneous of degree $2n$, $\pi_{n/2}(:\Omega_{i_r, j_0} \tilde{D}_{I_r,J'}:)$ consists of a piece of degree $2n+2$ and a piece of degree $2n$ coming from all double contractions of $\Omega_{i_r, j_0}$ with terms in $\tilde{D}_{I_r,J'}$, which lower the degree by two. The component of $$\pi_{n/2}\bigg(\sum_{r=0}^{n}  :\Omega_{i_r, j_0} \tilde{D}_{I_r,J'}:\bigg)\in \cV_{n/2}$$ in degree $2n+2$ must vanish since this sum corresponds to the relation $d_{I,J}$. The component of $:\Omega_{i_r,j_0}\tilde{D}_{I_r,J'}:$ in degree $2n$ is $$S_r =  (-1)^{i_r} \bigg(\sum_{k} \frac{\tilde{D}_{I_{r,k},J'}}{i_k +i_r+ 1}  + \sum_{l}  \frac{\tilde{D}_{I_r,J'_l}}{j_l + i_r+ 1} \bigg) + (-1)^{j_0+1} \bigg(\sum_{k} \frac{\tilde{D}_{I_{r,k},J'}}{i_k +j_0+ 1}  +\sum_{l}  \frac{\tilde{D}_{I_r,J'_l}}{j_l + j_0+ 1}  \bigg).$$ 
 
In this notation, for $k=0,\dots,n,$ and $k\neq r$, $I_{r,k}$ is obtained from $I_r = (i_0,\dots, \widehat{i_r},\dots, i_{n})$ by replacing the entry $i_k$ with $i_k+ i_r+j_0+1$. Similarly, for $l=1,\dots, n$, $J'_l$ is obtained from $J' = (j_1,\dots, j_{n})$ by replacing $j_l$ with $j_l + i_r+j_0 + 1$. It follows that \begin{equation} \label{usefuliiodd} \pi_{n/2}\bigg(\sum_{r=0}^{n}  :\Omega_{i_r, j_0} \tilde{D}_{I_r,J'}:\bigg)= \pi_{n/2}\bigg(\sum_{r=0}^{n} S_r\bigg).\end{equation}

Combining \eqref{usefuliodd} and \eqref{usefuliiodd}, we can regard $$\sum_{r=0}^{n} \sum_{i=1}^{n} :\Omega_{i_r, j_0} \tilde{D}^{2i}_{I_r,J'}: - \sum_{r=0}^{n} S_r$$ as a decomposition of $D_{I,J}$ of the form $D_{I,J} = \sum_{k=1}^{n+1} D^{2k}_{I,J}$ where the leading term $D^{2n+2}_{I,J} = \sum_{r=0}^{n} :\Omega_{i_r, j_0} \tilde{D}^{2n}_{I_r,J'}:$. It follows that $R_n(I,J)$ is the negative of the sum of the terms $R_{n-1}(K,L)$ corresponding to each $\tilde{D}_{K,L}$ appearing in $\sum_{r=0}^{n}S_r$. We obtain

\begin{equation} \label{recursionodd} \begin{split} R_n(I,J) = -\sum_{r=0}^{n} (-1)^{i_r} \bigg(\sum_{k} \frac{R_{n-1}(I_{r,k},J')}{i_k +i_r+ 1}  + \sum_{l}  \frac{R_{n-1}(I_r,J'_l)}{j_l + i_r+ 1} \bigg) \\ -\sum_{r=0}^{n} (-1)^{j_0+1} \bigg(\sum_{k} \frac{R_{n-1}(I_{r,k},J')}{i_k +j_0+ 1}  +\sum_{l}  \frac{R_{n-1}(I_r,J'_l)}{j_l + j_0+ 1}  \bigg).\end{split} \end{equation}

Suppose that all entries of $I$ are even and all entries of $J$ are odd. Each term of the form $R_{n-1}(K,L)$ appearing in \eqref{recursionodd} has the property that all entries of $K$ are even and all entries of $L$ are odd, so we may restrict ourselves to elements with this property. In the case $n=1$, $i_0, i_1$ even and $j_0, j_1$ odd, a calculation shows that $R_1(I,J)$ is given by
$$\frac{1}{1 + i_0 + i_1} + \frac{1}{2 (1 + i_0 + j_0)} +\frac{1}{2 (1 + i_1 + j_0)} +\frac{1}{2 (1 + i_0 + j_1)} + \frac{1}{2 (1 + i_1 + j_1)} +\frac{1}{1 + j_0 + j_1},$$
and in particular is nonzero. By induction on $n$, it is immediate from \eqref{recursionodd} that $R_n(I,J) \neq 0$ whenever $I$ consists of even entries and $J$ consists of odd entries. Specializing to the case $I = (0,\dots,0)$ and $J = (1,\dots, 1)$, we see that $R_0 \neq 0$. Combined with Lemma \ref{lemnonzerodd}, we obtain

\begin{thm} \label{descriptionofvnodd} For all $n\geq 1$, $\cV_{n/2}$ has a minimal strong generating set $\{w^1, w^3, \dots, w^{2n-1}\}$, and is therefore of type $\cW(2,4,\dots, 2n)$. \end{thm}

\section{The case of $\cV_{-n+1/2}$ for $n\geq 1$}
For each integer $n\geq 1$, $\cV_{-n+1/2}$ has a free field realization given by Proposition 14.2 of \cite{KWY}. First, there is an action of the level $-1$ affine vertex superalgebra of $\go\gs\gp(1,2n)$ on $\cS(n) \otimes \cF(1)$, and the action of the horizontal subalgebra $\go\gs\gp(1,2n)$ integrates to an action of the Lie supergroup $\text{Osp}(1,2n)$. 

\begin{thm} We have an isomorphism $\cV_{-n+1/2}\rightarrow (\cS(n) \otimes \cF(1))^{\text{Osp}(1,2n)}$ given by
\begin{equation} \label{bgrealizationi} w^{2m+1} \mapsto \frac{1}{2} \sum_{i=1}^n \big(:\beta^{i}\partial^{2m+1} \gamma^{i}:- :\partial^{2m+1} \beta^{i} \gamma^{i}: \big) - \frac{1}{2} :\phi \partial^{2m+1} \phi: ,\qquad m\geq 0.\end{equation} 
\end{thm}
This map preserves conformal structures and is a morphism in the category $\cR$. We have an isomorphism of $\text{Osp}(1,2n)$-modules $\cS(n) \otimes \cF(1) \cong \text{gr}(\cS(n) \otimes \cF(1))$ and an isomorphism of graded supercommutative rings
\begin{equation} \label{griossergeev} \text{gr}((\cS(n) \otimes \cF(1))^{\text{Osp}(1,2n)}) \cong \text{gr}(\cS(n) \otimes \cF(1))^{\text{Osp}(1,2n)}\cong R,\end{equation} where $R= \big (\text{Sym} \bigoplus_{k\geq 0} U_k\big) ^{\text{Osp}(1,2n)}$ and $U_k$ is a copy of the standard $\text{Osp}(1,2n)$-module $\mathbb{C}^{2n|1}$. The even subspace of  $U_k$ is spanned by $\beta^i_k, \gamma^i_k$, and the odd subspace is spanned by $\phi_k$, where $\beta^i_k, \gamma^i_k, \phi_k$ are the images of $\partial^k \beta^i, \partial^k \gamma^i, \partial^k \phi$ in $\text{gr}(\cS(n) \otimes \cF(1))$. The generators and relations for $R$ are given by Theorem 1.3 of \cite{SI} and Theorem 4.5 of \cite{SII}, respectively.
 
\begin{thm} \label{sergeev} For $k\geq 0$, let $U_k$ be copy of the standard $\text{Osp}(1|2n)$-module $\mathbb{C}^{2n|1}$, with even subspace spanned by $\{x_{k,i}, y_{k,i}|\ i=1,\dots, n\}$ and odd subspace spanned by $z_k$. Then $R =  (\text{Sym} \bigoplus_{k\geq 0}  U_k) ^{\text{Osp}(1,2n)}$ is generated by the quadratics
$$q_{a,b} = \frac{1}{2} \sum_{i=1}^n (x_{i,a} y_{i,b}- x_{i,b} y_{i,a})  -\frac{1}{2} z_{a} z_{b}\qquad a,b\geq 0.$$ Note that $q_{a,b} = -q_{b,a}$. Let $Q_{a,b}$ be commuting indeterminates satisfying $Q_{a,b} = -Q_{b,a}$. The kernel $I_n$ of the map $$\mathbb{C}[Q_{a,b}] \ra R,\qquad Q_{a,b} \mapsto q_{a,b}$$ is generated by polynomials $p_I$ of degree $2n+2$ in the variables $Q_{a,b}$ corresponding to a rectangular Young tableau of size $2\times (2n+2)$, filled by entries from a standard sequence $I$ of length $4n+4$ from the set of indices $\{0,1,2,\dots\}$. The entries must strictly increase along rows and weakly increase along columns. \end{thm}

Under \eqref{griossergeev}, the generators $q_{a,b}$ correspond to strong generators
$$\omega_{a,b} \mapsto \frac{1}{2} \sum_{i=1}^n \big(: \partial^a\beta^{i}\partial^{b} \gamma^{i}:- :\partial^{b} \beta^{i} \partial^a \gamma^{i}: \big) - \frac{1}{2} :\partial^a \phi \partial^{b} \phi:$$ for $\cV_{-n+1/2}$. In this notation, $w^{2m+1} = \omega_{0,2m+1}$. Letting $A_m$ denote the vector space spanned by $\{\omega_{a,b}|\ a+b=m\}$, we have a decomposition
$$A_{2m+1} = \partial (A_{2m}) \oplus \bra w^{2m+1} \ket = \partial^2 (A_{2m-1}) \oplus \bra w^{2m+1}\ket.$$ There is a corresponding strong generating set $\{\Omega_{a,b}\}$ for the vacuum module $\cM^+_{-n+1/2}$ satisfying $\pi_{-n+1/2}(\Omega_{a,b}) = \omega_{a,b}$, and we also denote the span of $\{\Omega_{a,b}|\ a+b=m\}$ by $A_m$. We have the following analogue of Lemma \ref{lemddef}.

\begin{lemma} \label{lemddefsergeev} For each relation $p_I$ as above, there exists a unique element \begin{equation} \label{ddefoddsergeev} P_I  \in (\cM^+_{-n+1/2})_{(4n+4)}\cap \cI_{-n+1/2}\end{equation} of weight $2n+2 +\sum_{a=0}^{4n+3} i_a$, satisfying \begin{equation}\label{uniquedijoddsergeev} \phi_{4n+4}(P_I)= p_I.\end{equation} These elements generate $\cI_{-n+1/2}$ as a vertex algebra ideal.
\end{lemma}

Each $P_I$ can be written in the form 
\begin{equation}\label{decompsergeev} P_I = \sum_{k=1}^{2n+2} P^{2k}_I,\end{equation} where $P^{2k}_I$ is a normally ordered polynomial of degree $k$ in the elements $\Omega_{a,b}$. The term $P^2_I$ lies in $A_m$ for $m =2 n+1 + \sum_{a=0}^{4n+3} i_a$. For odd $m\geq 1$ we have the projection $$\text{pr}_{m}: A_{m}\ra \bra W^{m}\ket.$$ For all $I$ such that  $m = 2n+1+ \sum_{a=0}^{4n+3} i_a$ is odd, define the remainder \begin{equation}\label{defofrijodd} R_I = \text{pr}_m(P^2_I),\end{equation} which is independent of the choice of decomposition \eqref{decompsergeev}.

The relation of minimal weight corresponds to the $2\times (2n+2)$ tableau where both rows are labelled by $(0,1,\dots,2n+1)$, and therefore has weight $4n^2+8n+4$. We denote this relation by $P_0$ and we denote its remainder by $R_0$. The condition $R_0 \neq 0$ is equivalent to the existence of a decoupling relation
\begin{equation} \label{decouposp} w^{4n^2+8n+3} = Q(w^1, w^3, \dots, w^{4n^2+8n+1}),\end{equation} where $Q$ is a normally ordered polynomial in $w^1, w^3,  \dots, w^{4n^2+8n+1}$, and their derivatives. As in Lemma \ref{lemnonzeroii}, by applying $w^3 \circ_1$ repeatedly to this relation, we obtain relations $$w^{2m+1} = Q_m(w^1, w^3, \dots, w^{4n^2+8n+1})$$ for all $m>2n^2+4n+1$. 

\begin{conj} \label{ospconj} For all $n\geq 1$, $R_0 \neq 0$. Equivalently, $\cV_{-n+1/2}$ has a minimal strong generating set $\{w^1, w^3, \dots, w^{4n^2+8n+1}\}$, and is therefore of type $\cW(2,4,\dots, 4n^2+8n+2)$. \end{conj}

We are unable to prove this conjecture because a recursive formula for $R_I$ is currently out of reach, but for $n=1$ we can prove it by computer calculation. In this case, the relation $p_0$ of minimal weight $16$ is given by 
\begin{equation} \label{defofpo} q_{0,1}^2  q_{2,3}^2 + q_{0,2}^2  q_{1,3}^2+  q_{0,3}^2  q_{1,2}^2  - 2  q_{0,2}  q_{0,3}  q_{1,2}   q_{1,3}  + 2  q_{0,1}  q_{0,3}  q_{1,2}   q_{2,3} - 2  q_{0,1}   q_{0,2}  q_{1,3}   q_{2,3}. \end{equation}
In the Appendix, we will write down the corresponding relation $P_0$ explicitly, and we will see that $R_0 =  \frac{109}{56000} W^{15}$. This implies

\begin{thm} $\cV_{-1/2}$ has a minimal strong generating set $\{w^1, w^3, \dots, w^{13}\}$, and in particular is of type $\cW(2,4,\dots, 14)$. \end{thm}

\section{Representation theory of $\cV_{-n}$ and $\cV_{n/2}$}
The basic tool in the representation theory of vertex algebras is the {\it Zhu functor} \cite{Z}. Given a vertex algebra $\cW$ with weight grading $\cW = \bigoplus_{n\in\mathbb{Z}} \cW_n$, this functor attaches to $\cW$ an associative algebra $A(\cW)$, together with a surjective linear map $\pi_{\text{Zhu}}:\cW\ra A(\cW)$. For $a\in \cW_{m}$ and $b\in\cW$, define
\begin{equation}\label{defzhu} a*b = \text{Res}_z \bigg (a(z) \frac{(z+1)^{m}}{z}b\bigg),\end{equation} and extend $*$ by linearity to a bilinear operation $\cW\otimes \cW\ra \cW$. Let $O(\cW)$ denote the subspace of $\cW$ spanned by elements of the form \begin{equation}\label{zhuideal} a\circ b = \text{Res}_z \bigg (a(z) \frac{(z+1)^{m}}{z^2}b\bigg)\end{equation} where $a\in \cW_m$, and let $A(\cW)$ be the quotient $\cW/O(\cW)$, with projection $\pi_{\text{Zhu}}:\cW\ra A(\cW)$. Then $O(\cW)$ is a two-sided ideal in $\cW$ under the product $*$, and $(A(\cW),*)$ is a unital, associative algebra. The assignment $\cW\mapsto A(\cW)$ is functorial, and if $\cI \subset \cW$ is a vertex algebra ideal, $A(\cW/\cI)\cong A(\cW)/ I$, where $I = \pi_{\text{Zhu}}(\cI)$.

If $\cW$ is strongly generated by homogeneous elements $\{\alpha_i|\ i\in I\}$, $A(\cW)$ is generated by $\{ a_i = \pi_{\text{Zhu}}(\alpha_i)|\ i\in I\}$. A $\mathbb{Z}_{\geq 0}$-graded $\cW$-module $M = \bigoplus_{n\geq 0} M_n$ is called {\it admissible} if for every $a\in\cW_m$, $a(n) M_k \subset M_{m+k -n-1}$, for all $n\in\mathbb{Z}$. Given $a\in\cW_m$, $a(m-1)$ acts on each $M_k$. The subspace $M_0$ is then an $A(\cW)$-module with action $\pi_{\text{Zhu}}(a)\mapsto a(m-1) \in \text{End}(M_0)$. In fact, $M\mapsto M_0$ provides a one-to-one correspondence between irreducible, admissible $\cW$-modules and irreducible $A(\cW)$-modules. If $A(\cW)$ is commutative, all its irreducible modules are one-dimensional, and the corresponding $\cW$-modules $M = \bigoplus_{n\geq 0} M_n$ are cyclic and generated by any nonzero $v\in M_0$. Accordingly, we call such a module a {\it highest-weight module} for $\cW$, and we call $v$ a {\it highest-weight vector}.

\begin{thm} \label{abelianzhu} For all $c\in \mathbb{C}$, $A(\cV_{c})$ is commutative, so all irreducible, admissible $\cV_{c}$-modules are highest-weight modules.\end{thm}

\begin{proof} Since $A(\cV_{c})$ is a homomorphic image of $A(\cM^+_c)$, it suffices to show that $A(\cM^+_c)$ is commutative. But this is clear since the generators of $\cM^+_c$ are linear combinations of the generators of $\cM_c$, and $A(\cM_c)$ is known to be commutative \cite{FKRW}. \end{proof}

Since $\cM^+_{c}$ is freely generated by $W^1,W^3, \dots$ it follows that $A(\cM^+_{c})$ is the polynomial algebra $\mathbb{C}[A^1,A^3,\dots]$, where $A^{2m+1} = \pi_{\text{Zhu}}(W^{2m+1})$. Moreover, $$A(\cV_{c}) \cong  \mathbb{C}[a^1,a^3,\dots] / I_{c},$$ where $a^{2m+1} = \pi_{\text{Zhu}}(w^{2m+1})$ and $I_{c} = \pi_{\text{Zhu}}(\cI_{c})$. We have a commutative diagram
\begin{equation}\label{commdiag} \begin{array}[c]{ccc}
\cM^+_{c} &\stackrel{\pi_{c}}{\rightarrow}& \cV_{c}  \\
\downarrow\scriptstyle{\pi_{\text{Zhu}}}&&\downarrow\scriptstyle{\pi_{\text{Zhu}}}\\
A(\cM^+_{c}) &\stackrel{A(\pi_{c})}{\rightarrow}& A(\cV_{c})
\end{array} .\end{equation} 
Since $A(\cV_{-n})$ is generated by $a^1, a^3,\dots,a^{2n^2+4n-1}$, $$A(\cV_{-n}) \cong \mathbb{C}[a^1,a^3,\dots, a^{2n^2+4n-1}]/ I_{-n},$$ where $I_{-n}$ is now regarded as an ideal inside $\mathbb{C}[a^1,a^3,\dots, a^{2n^2+4n-1}]$. The corresponding variety $V(I_{-n})$, which parametrizes the irreducible, admissible $\cV_{-n}$-modules, is a proper, closed subvariety of $\mathbb{C}^{n^2+2n}$. Similarly, $$A(\cV_{n/2}) \cong \mathbb{C}[a^1,a^3,\dots, a^{2n-1}]/ I_{n/2},$$ where $I_{n/2}$ is an ideal inside $\mathbb{C}[a^1,a^3,\dots, a^{2n-1}]$, and the variety $V(I_{n/2})$ parametrizing the irreducible, admissible $\cV_{n/2}$-modules is a proper, closed subvariety of $\mathbb{C}^{n}$. The proof that these varieties are proper is analogous to the proof of Theorem 5.1 of \cite{LI} and is omitted. It amounts to constructing nontrivial relations among the generators $A(\cV_{-n})$ and $A(\cV_{n/2})$, which come from normally ordered relations in $\cV_{-n}$ and $\cV_{n/2}$, respectively. In particular, neither $\cV_{-n}$ nor $\cV_{n/2}$ is {\it freely} generated by the elements given by Theorems \ref{descriptionofvn} and \ref{descriptionofvnodd}.

\section{The Hilbert problem for $\cS(n)$ and $\cF(n)$}\label{secgeneral}
For a simple, strongly finitely generated vertex algebra $\cA$ and a reductive group $G\subset \text{Aut}(\cA)$, the {\it Hilbert problem} asks whether $\cA^G$ is strongly finitely generated. Using our results on the structure of $\cV_{-n}$ and $\cV_{n/2}$ we give a complete solution to this problem for any $G$ when $\cA$ is either the $\beta\gamma$-system $\cS(n)$ or the free fermion algebra $\cF(n)$. This generalizes our earlier study \cite{LII} in which we established the strong finite generation of $\cS(n)^G$ and $\cF(2n)^G$ for $G\subset \text{GL}(n)$.

By Theorem 13.2 and Corollary 14.2 of \cite{KWY}, $\cS(n)$ has a decomposition \begin{equation}\label{dlmdecomp} \cS(n) \cong \bigoplus_{\nu\in H} L(\nu)\otimes M^{\nu},\end{equation} where $H$ indexes the irreducible, finite-dimensional representations $L(\nu)$ of $\text{Sp}(2n)$, and the $M^{\nu}$'s are inequivalent, irreducible, highest-weight $\cV_{-n}$-modules. The modules $M^{\nu}$ above have an integrality property; the eigenvalues of $$\{w^{2m+1}(2m+1)|\ m\geq 0\}$$ on the highest-weight vectors $f_{\nu}$ are all integers. These modules therefore correspond to certain rational points on the variety $V(I_{-n})$. Recall that $\cS(n)\cong \text{gr}(\cS(n))$ as $\text{Sp}(2n)$-modules, and $$\text{gr}(\cS(n)^G )\cong (\text{gr}(\cS(n))^G \cong (\text{Sym} \bigoplus_{k\geq 0} U_k)^G = R$$ as commutative algebras, where $U_k\cong \mathbb{C}^{2n}$. For all $p\geq 1$, $\text{GL}(p)$ acts on $\bigoplus_{k =0}^{p-1} U_k $ and commutes with the action of $G$. There is an induced action of $\text{GL}(\infty) = \lim_{p\ra \infty} \text{GL}(p)$ on $\bigoplus_{k\geq 0} U_k$ which commutes with $G$, so $\text{GL}(\infty)$ acts on $R$. Elements $\sigma \in \text{GL}(\infty)$ are known as {\it polarization operators}, and given $f\in R$, $\sigma f$ is known as a polarization of $f$. The following fundamental result of Weyl appears as Theorem 2.5A of \cite{We}.
\begin{thm} \label{weylfinite} $R$ is generated by the polarizations of any set of generators for $(\text{Sym} \bigoplus_{k = 0} ^{2n-1} U_k)^G$. Since $G$ is reductive, $(\text{Sym} \bigoplus_{k = 0} ^{2n-1} U_k)^G$ is finitely generated, so there exists a finite set $\{f_1,\dots, f_r\}$, whose polarizations generate $R$. \end{thm}

\begin{lemma} \label{ordfg} $\cS(n)^G$ has a strong generating set which lies in the direct sum of finitely many irreducible $\cV_{-n}$-modules.
\end{lemma}

\begin{proof} Each of the modules $L(\nu)$ appearing in \eqref{dlmdecomp} is a $G$-module, and since $G$ is reductive, it has a decomposition $L(\nu) =\bigoplus_{\mu\in H^{\nu}} L(\nu)_{\mu}$. Here $\mu$ runs over a finite set $H^{\nu}$ of irreducible, finite-dimensional representations $L(\nu)_{\mu}$ of $G$, possibly with multiplicity. We obtain a refinement of \eqref{dlmdecomp},
\begin{equation}\label{decompref} \cS(n) \cong \bigoplus_{\nu\in H} \bigoplus_{\mu\in H^{\nu}} L(\nu)_{\mu} \otimes M^{\nu}.\end{equation}
Under the linear isomorphism $\cS(n)^G\cong R$, let $f_i(z)$ and $(\sigma f_i)(z)$ correspond to $f_i$ and $\sigma f_i$, respectively, for $i=1,\dots, r$. By Lemma \ref{reconlem}, the set $$\{(\sigma f_i)(z)\in \cS(n)^G|\ i=1,\dots,r,\ \sigma\in \text{GL}(\infty)\},$$ is a strong generating set for $\cS(n)^G$. Clearly $f_1(z),\dots, f_r(z)$ must lie in a finite direct sum
\begin{equation}\label{newsummation} \bigoplus_{j=1}^t L(\nu_j)\otimes M^{\nu_j}\end{equation} of the modules appearing in (\ref{dlmdecomp}). By enlarging the collection $f_1(z),\dots,f_r(z)$ if necessary, we may assume without loss of generality that each $f_i(z)$ lies in a single representation of the form $L(\nu_j)\otimes M^{\nu_j}$. Moreover, we may assume that $f_i(z)$ lies in a trivial $G$-submodule $L(\nu_j)_{\mu_0} \otimes M^{\nu_j}$, where $\mu_0$ denotes the trivial, one-dimensional $G$-module. (In particular, $L(\nu_j)_{\mu_0}$ is one-dimensional). Since the actions of $\text{GL}(\infty)$ and $\text{Sp}(2n)$ on $\cS(n)$ commute, $(\sigma f_i)(z)\in L(\nu_j)_{\mu_0}\otimes M^{\nu_j}$ for all $\sigma\in \text{GL}(\infty)$. \end{proof}

\begin{cor} $\cS(n)^G$ is a finitely generated vertex algebra.\end{cor}

\begin{proof} Since $\cS(n)^G$ is strongly generated by $\{ (\sigma f_i)(z)|\ i=1,\dots,r,\  \sigma\in \text{GL}(\infty)\}$, and each $M^{\nu_j}$ is an irreducible $\cV_{-n}$-module, $\cS(n)^G$ is generated by $f_1(z),\dots,f_r(z)$ as an algebra over $\cV_{-n}$. Since $\cV_{-n}$ is generated by $w^3$ as a vertex algebra, $\cS(n)^G$ is finitely generated. \end{proof}

We need a fact about representations of associative algebras which can be found in \cite{LII}. Let $A$ be an associative $\mathbb{C}$-algebra, and let $W$ be a linear representation of $A$, via an algebra homomorphism $\rho: A\ra \text{End}(W)$. Regarding $A$ as a Lie algebra with commutator as bracket, let $\rho_{\text{Lie}}:A\ra \text{End}(W)$ denote the map $\rho$, regarded now as a Lie algebra homomorphism. Clearly $\rho_{\text{Lie}}$ extends to a Lie algebra homomorphism $\hat{\rho}_{\text{Lie}}: A\ra \text{End}(\text{Sym}(W))$, where $\hat{\rho}_{\text{Lie}}(a)$ acts by derivation on each $\text{Sym}^d(W)$: $$\hat{\rho}_{\text{Lie}}(a)( w_1\cdots w_d) = \sum_{i=1}^d w_1 \cdots  \hat{\rho}_{\text{Lie}}(a)(w_i)  \cdots  w_d.$$ This extends to an algebra homomorphism $U(A)\ra \text{End}(\text{Sym}(W))$ which we also denote by $\hat{\rho}_{\text{Lie}}$. The following result appears as Lemma 3 of \cite{LII}.

\begin{lemma} \label{first} Given $\mu \in U(A)$ and $d\geq 1$, let $\Phi^d_{\mu}  = \hat{\rho}_{\text{Lie}}(\mu) \big|_{\text{Sym}^d(W)} \in \text{End}(\text{Sym}^d(W))$. Let $E$ denote the subspace of $\text{End}(\text{Sym}^d(W))$ spanned by $\{\Phi^d_{\mu}|\ \mu\in U(A)\}$, which has a filtration $$E_1\subset E_2\subset \cdots,\qquad  E = \bigcup_{r\geq 1} E_r.$$ Here $E_r$ is spanned by $\{\Phi^d_{\mu}|\  \mu \in U(A),\ \text{deg}(\mu) \leq r\}$. Then $E = E_d$.\end{lemma}

\begin{cor} \label{firstcor} Let $f\in \text{Sym}^d(W)$, and let $M\subset \text{Sym}^d(W)$ be the cyclic $U(A)$-module generated by $f$. Then $\{\hat{\rho}_{\text{Lie}}(\mu)(f)|\ \mu\in U(A),\ \text{deg}(\mu)\leq d\}$ spans $M$.\end{cor}

Recall the Lie algebra $\cP^+\subset \hat{\cD}^+$ generated by the modes $$\{w^{2m+1}(k)|\ m\geq 0, \ k\geq 0\}.$$ Note that $\cP^+$ has a decomposition $$\cP^+ = \cP^+_{<0} \oplus \cP^+_0\oplus \cP^+_{>0},$$ where $\cP^+_{<0}$, $\cP^+_0$, and $\cP^+_{>0}$ are the Lie algebras spanned by $$\{w^{2m+1}(k)|\ 0\leq k< 2m+1\},\qquad \{w^{2m+1}(2m+1)\},\qquad \{w^{2m+1}(k)|\ k>2m+1\},$$ respectively. Clearly $\cP^+$ preserves the filtration on $\cS(n)$, so each element of $\cP^+$ acts by a derivation of degree zero on $\text{gr}(\cS(n))$. 

Let $\cM$ be an irreducible, highest-weight $\cV_{-n}$-submodule of $\cS(n)$ with generator $f(z)$, and let $\cM'\subset \cM$ denote the $\cP^+$-submodule generated by $f(z)$. Since $f(z)$ has minimal weight among elements of $\cM$ and $\cP^+_{>0}$ lowers weight, $f(z)$ is annihilated by $\cP^+_{>0}$. Moreover, $\cP^+_0$ acts diagonalizably on $f(z)$, so $f(z)$ generates a one-dimensional $\cP^+_0\oplus \cP^+_{>0}$-module. By the Poincar\'e-Birkhoff-Witt theorem, $\cM'$ is a quotient of $$U(\cP^+)\otimes_{U(\cP^+_0\oplus \cP^+_{>0})} \mathbb{C} f(z),$$ and in particular is a cyclic $\cP^+_{<0}$-module with generator $f(z)$. Suppose that $f(z)$ has degree $d$, that is, $f(z)\in \cS(n)_{(d)} \setminus \cS(n)_{(d-1)}$. Since $\cP^+$ preserves the filtration on $\cS(n)$, and $\cM$ is irreducible, the nonzero elements of $\cM'$ lie in $\cS(n)_{(d)} \setminus \cS(n)_{(d-1)}$. Therefore, the projection $\cS(n)_{(d)}\ra \cS(n)_{(d)}/\cS(n)_{(d-1)} \subset \text{gr}(\cS(n))$ restricts to an isomorphism of $\cP^+$-modules \begin{equation}\label{isopmod} \cM'\cong \text{gr}(\cM')\subset \text{gr}(\cS(n)).\end{equation} By Corollary \ref{firstcor}, $\cM'$ is spanned by elements of the form $$\{w^{2l_1+1}(k_1)\cdots w^{2l_r+1}(k_r) f(z) |\ w^{2l_i+1}(k_i)\in \cP^+_{<0},\ r\leq d\}.$$

The next result is analogous to Lemma 7 of \cite{LII}, and the proof is the same.
\begin{lemma} \label{fourth} Let $\cM$ be an irreducible, highest-weight $\cV_{-n}$-submodule of $\cS(n)$ with highest-weight vector $f(z)$ of degree $d$. Let $\cM'$ be the corresponding $\cP^+$-module generated by $f(z)$, and let $f$ be the image of $f(z)$ in $\text{gr}(\cS(n))$, which generates $M = \text{gr}(\cM')$ as a $\cP^+$-module. Fix $m$ so that $f\in \text{Sym}^d(W_m)$, where $W_m\subset \text{gr}(\cS(n))$ has basis $\{\beta^i_j, \gamma^i_j|\ 1\leq  i \leq n,\ 0\leq j\leq m\}$. Then $\cM'$ is spanned by $$\{w^{2l_1+1}(k_1) \cdots w^{2l_r+1}(k_r) f(z)|\ w^{2l_i+1}(k_i)\in \cP^+_{<0},\ \  r\leq d,\ \ 0\leq k_i \leq 2m+1\}.$$\end{lemma}

We order the elements $w^{2l+1}(k)\in \cP^+_{<0}$ as follows: $w^{2l_1+1}(k_1) > w^{2l_2+1}(k_2)$ if $l_1>l_2$, or $l_1=l_2$ and $k_1<k_2$. Then Lemma \ref{fourth} can be strengthened as follows: $\cM'$ is spanned by elements of the form $w^{2l_1+1}(k_1)\cdots w^{2l_r+1}(k_r) f(z)$ with \begin{equation}\label{shapealpha} w^{2l_i+1}(k_i)\in \cP^+_{<0},\quad  r\leq d,\quad  0\leq k_i\leq 2m+1,\quad  w^{2l_1+1}(k_1)\geq \cdots \geq w^{2l_r+1}(k_r).\end{equation}

As in \cite{M}, given an irreducible $\cV_{-n}$-submodule $\cM$ of $\cS(n)$, define $C_1(\cM)$ to be the subspace of $\cM$ spanned by elements of the form $$:a(z) b(z):, \qquad a(z)\in \bigoplus_{k>0}\cV_{-n}[k],\qquad b(z)\in \cM.$$ Following Miyamoto's definition in \cite{M}, we say that $\cM$ is $C_1$-cofinite as a $\cV_{-n}$-module if $\cM / C_1(\cM)$ is finite-dimensional.

\begin{lemma} \label{fifth} Every irreducible $\cV_{-n}$-submodule $\cM$ of $\cS(n)$ is $C_1$-cofinite. 
\end{lemma}
\begin{proof} The argument is the same as the proof of Lemma 8 of \cite{LII}, and is omitted. Note that Lemma 8 of \cite{LII} implies that every irreducible $\cW_{1+\infty,-n}$-submodule of $\cS(n)$ is $C_1$-cofinite, but this terminology was not used. \end{proof}

\begin{cor} \label{sixth} Let $\cM$ be an irreducible $\cV_{-n}$-submodule of $\cS(n)$. Given a subset $S\subset \cM$, let $\cM_S\subset \cM$ denote the subspace spanned by elements of the form $$:\omega_1(z)\cdots \omega_t(z) \alpha(z):,\qquad \omega_j(z)\in \cV_{-n},\qquad \alpha(z)\in S.$$ There exists a finite set $S\subset \cM$ such that $\cM = \cM_S$.\end{cor}

\begin{thm} \label{sfg} For any reductive group $G$ of automorphisms of $\cS(n)$, $\cS(n)^G$ is strongly finitely generated.\end{thm}

\begin{proof} By Lemma \ref{ordfg}, we can find $f_1(z),\dots, f_r(z) \in \cS(n)^G$ such that $\text{gr}(\cS(n))^G$ is generated by the corresponding polynomials $f_1,\dots, f_r\in \text{gr}(\cS(n))^G$, together with their polarizations. We may assume that each $f_i(z)$ lies in an irreducible, highest-weight $\cV_{-n}$-module $\cM_i$ of the form $L(\nu)_{\mu_0}\otimes M^{\nu}$, where $L(\nu)_{\mu_0}$ is a trivial, one-dimensional $G$-module. Furthermore, we may assume that $f_1(z),\dots, f_r(z)$ are highest-weight vectors for the action of $\cV_{-n}$. For each $\cM_i$, choose a finite set $S_i \subset \cM_i$ such that $\cM_i = (\cM_i)_{S_i}$, and define $$ S=\{w^1,w^3,\dots, w^{2n^2+4n-1} \} \cup \big(\bigcup_{i=1}^r S_i \big).$$ Since $\{w^1,w^3,\dots, w^{2n^2+4n-1}\}$ strongly generates $\cV_{-n}$ and $\bigoplus_{i=1}^r \cM_i$ contains a strong generating set for $\cS(n)^G$, $S$ strongly generates $\cS(n)^G$. \end{proof}

We sketch the proof of the analogous results for $\cF(n)^G$ where $G\subset \text{O}(n)$ is a reductive group. By Theorem 11.2 and Corollary 14.2 of \cite{KWY}, $\cF(n)$ has a decomposition $$ \cF(n) \cong \bigoplus_{\mu\in H'} L(\mu)\otimes M^{\mu},$$ where $H'$ indexes the irreducible, finite-dimensional representations $L(\mu)$ of $\text{O}(n)$, and the $M^{\mu}$'s are inequivalent, irreducible, highest-weight $\cV_{n/2}$-modules. Since $G$ is reductive, $\cF(n)^G$ is a direct sum of $M^{\mu}$'s as well. By Theorem 6.4 of \cite{CL}, which is an analogue of Theorem \ref{weylfinite} for odd variables, $\cF(n)^G$ has a strong generating set which lies in a finite sum of these modules. Each irreducible $\cV_{n/2}$-module $\cM$ appearing in $\cF(n)$ has a finiteness property analogous to Lemma \ref{fourth}. Let $f(z)$ be the highest-weight vector of $\cM$, and suppose that $f(z)$ has degree $d$. Let $\cM'$ be the corresponding $\cP^+$-module generated by $f(z)$, and let $f$ be the image of $f(z)$ in $\text{gr}(\cF(n))$, which generates $M = \text{gr}(\cM')$ as a $\cP^+$-module. Fix $m$ so that $f\in \bigwedge^d(W_m)$. Then $\cM'$ is spanned by $$\{w^{2l_1+1}(k_1) \cdots w^{2l_r+1}(k_r) f(z)|\ w^{2l_i+1}(k_i)\in \cP^+_{<0},\ \  r\leq d,\ \ 0\leq k_i \leq 2m+1\}.$$ This implies the $C_1$-cofiniteness property of all irreducible $\cV_{n/2}$-submodules $\cM$ of $\cF(n)$. In particular, there exists a finite set $S$ such that $\cM = \cM_S$, where $\cM_S$ is spanned by $$:\omega_1(z)\cdots \omega_t(z) \alpha(z):,\qquad \omega_j(z)\in \cV_{n/2},\qquad \alpha(z)\in S.$$ Combined with Theorem \ref{descriptionofvnodd}, this implies
\begin{thm} \label{sfgodd} For any reductive group $G$ of automorphisms of $\cF(n)$, $\cF(n)^G$ is strongly finitely generated.\end{thm} In particular, if $\cM_1,\dots, \cM_r$ is a set of irreducible $\cV_{n/2}$-submodules of $\cF(n)^G$ whose direct sum contains a strong generating set for $\cF(n)^G$, and $\cM_i = (\cM_i)_{S_i}$ for finite sets $S_i$, we can take $S=\{w^1,w^3,\dots, w^{2n-1} \} \cup \big(\bigcup_{i=1}^r S_i \big)$ as our strong generating set for $\cF(n)^G$.

\section{Appendix}
In this Appendix, we give an explicit formula for the relation $P_0$ among the generators of $\cV_{-1/2}$, which corresponds to the classical relation \eqref{defofpo}. Define
$$P^8_0 =\  : \Omega_{0,1}   \Omega_{0,1}   \Omega_{2,3}   \Omega_{2,3}:  + 
  : \Omega_{0,2}   \Omega_{0,2}   \Omega_{1,3}   \Omega_{1,3}:  + : \Omega_{0,3}   \Omega_{0,3}   \Omega_{1,2}   \Omega_{1,2}: $$ $$- 2  : \Omega_{0,2}   \Omega_{0,3}   \Omega_{1,2}   \Omega_{1,3}:  + 2  : \Omega_{0,1}   \Omega_{0,3}   \Omega_{1,2}   \Omega_{2,3}: - 2  : \Omega_{0,1}   \Omega_{0,2}   \Omega_{1,3}   \Omega_{2,3}: , $$ which is clearly a normal ordering of \eqref{defofpo}. Next, we define the following quantum corrections of $P^8_0$.
 
$$P^6_0 =  - \frac{13}{84}  : \Omega_{0,1}   \Omega_{0,1}   \Omega_{2,9}: + \frac{11}{60}  : \Omega_{0,1}   \Omega_{0,1}   \Omega_{3,8}: - 
 \frac{2}{35}  : \Omega_{0,1}   \Omega_{0,2}   \Omega_{2,8}: + \frac{1}{12}  : \Omega_{0,1}   \Omega_{0,2}   \Omega_{3,7}: $$
  $$+ \frac{11}{30}  : \Omega_{0,1} \Omega_{0,3}   \Omega_{2,7}: - \frac{9}{20}  : \Omega_{0,1}   \Omega_{0,3}   \Omega_{3,6}: - 
 \frac{11}{28}  : \Omega_{0,1}   \Omega_{1,2}   \Omega_{2,7}: + \frac{1}{2}  : \Omega_{0,1}   \Omega_{1,2}   \Omega_{3,6}: $$ $$+ 
 \frac{1}{12}  : \Omega_{0,1}   \Omega_{1,3}   \Omega_{2,6}: - \frac{2}{15}  : \Omega_{0,1}   \Omega_{1,3}   \Omega_{3,5}: - 
 \frac{5}{12}  : \Omega_{0,1}   \Omega_{2,3}   \Omega_{1,6}: - \frac{1}{5}  : \Omega_{0,1}   \Omega_{2,3}   \Omega_{2,5}: $$ $$+ 
 \frac{11}{12}  : \Omega_{0,1}   \Omega_{2,3}   \Omega_{3,4}: + \frac{1}{35}  : \Omega_{0,2}   \Omega_{0,2}   \Omega_{1,8}: - 
 \frac{1}{15}  : \Omega_{0,2}   \Omega_{0,2}   \Omega_{3,6}: - \frac{11}{30}  : \Omega_{0,2}   \Omega_{0,3}   \Omega_{1,7}: $$ $$+ 
 \frac{7}{12}  : \Omega_{0,2}   \Omega_{0,3}   \Omega_{3,5}: + \frac{11}{28}  : \Omega_{0,2}   \Omega_{1,2}   \Omega_{1,7}: - 
 \frac{7}{10}  : \Omega_{0,2}   \Omega_{1,2}  \Omega_{3,5}: + \frac{1}{3}  : \Omega_{0,2}   \Omega_{1,3}   \Omega_{1,6}:$$ $$ - 
 \frac{1}{4}  : \Omega_{0,2}   \Omega_{1,3}  \Omega_{2,5}: - \frac{1}{12}  : \Omega_{0,2}   \Omega_{1,3}   \Omega_{3,4}: + 
 \frac{9}{20}  : \Omega_{0,2}   \Omega_{2,3}   \Omega_{1,5}: + \frac{9}{40}  : \Omega_{0,3}   \Omega_{0,3}   \Omega_{1,6}: $$ $$- 
 \frac{7}{24}  : \Omega_{0,3}   \Omega_{0,3}   \Omega_{2,5}: - \frac{11}{12}  : \Omega_{0,3}   \Omega_{1,2}   \Omega_{1,6}: + 
 \frac{19}{20}  : \Omega_{0,3}   \Omega_{1,2}   \Omega_{2,5}: + \frac{1}{3}  : \Omega_{0,3}   \Omega_{1,2}   \Omega_{3,4}: $$ $$- 
 \frac{11}{56}  : \Omega_{1,2}   \Omega_{1,2}   \Omega_{0,7}: + \frac{5}{8}  : \Omega_{1,2}   \Omega_{1,2}   \Omega_{3,4}: + 
 \frac{2}{15}  : \Omega_{0,3}   \Omega_{1,3}   \Omega_{1,5}: - \frac{1}{4}  : \Omega_{0,3}   \Omega_{1,3}   \Omega_{2,4}: $$ $$+ 
 \frac{1}{12}  : \Omega_{1,2}   \Omega_{1,3}   \Omega_{0,6}: - \frac{7}{12}  : \Omega_{0,3}   \Omega_{1,4}   \Omega_{2,3}: + 
\frac{ 5}{12}  : \Omega_{0,3}   \Omega_{2,3}   \Omega_{2,3}: - \frac{9}{20}  : \Omega_{1,2}   \Omega_{0,5}   \Omega_{2,3}: $$ $$- 
\frac{ 3}{4}  : \Omega_{1,2}   \Omega_{2,3}   \Omega_{2,3}: + \frac{7}{12}  : \Omega_{0,4}   \Omega_{1,3}   \Omega_{2,3}: - 
 \frac{1}{15}  : \Omega_{1,3}   \Omega_{1,3}   \Omega_{0,5}: + \frac{1}{3}  : \Omega_{1,3}   \Omega_{1,3}   \Omega_{2,3}:, $$

 $$P^4_0 =  \frac{19}{432}  : \Omega_{0,1}   \Omega_{1,12}: - \frac{113}{6720}  : \Omega_{0,1}   \Omega_{2,11}: - 
 \frac{569}{8640}  : \Omega_{0,1}   \Omega_{3,10}: + \frac{143}{1008}  : \Omega_{0,1}   \Omega_{4,9}: $$ $$ - 
 \frac{23}{700}  : \Omega_{0,1}   \Omega_{5,8}:  - \frac{559}{2016}  : \Omega_{0,1}   \Omega_{6,7}: - 
 \frac{151}{20160}  : \Omega_{0,2}   \Omega_{1,11}: - \frac{1}{252}  : \Omega_{0,2}   \Omega_{2,10}: $$ $$- 
 \frac{55}{576}  : \Omega_{0,2}   \Omega_{3,9}: + \frac{1}{420}  : \Omega_{0,2}   \Omega_{4,8}: + \frac{851}{2400}  : \Omega_{0,2}   \Omega_{5,7}: - 
 \frac{713}{6720}  : \Omega_{0,3}   \Omega_{1,10}: $$ $$- \frac{751}{20160}  : \Omega_{0,3}   \Omega_{2,9}: + 
 \frac{163}{672}  : \Omega_{0,3}   \Omega_{3,8}: - \frac{73}{1440}  : \Omega_{0,3}   \Omega_{4,7}: - \frac{49}{100}  : \Omega_{0,3}   \Omega_{5,6}: $$ $$- 
 \frac{163}{4032}  : \Omega_{1,2}   \Omega_{0,11}: + \frac{17}{336}  : \Omega_{1,2}   \Omega_{1,10}: + 
 \frac{1639}{10080}  : \Omega_{1,2}   \Omega_{2,9}: - \frac{227}{2240}  : \Omega_{1,2}   \Omega_{3,8}: $$ $$- 
 \frac{55}{168}  : \Omega_{1,2}   \Omega_{4,7}: + \frac{7}{32}  : \Omega_{1,2}   \Omega_{5,6}: + \frac{1}{60}  : \Omega_{0,4}   \Omega_{2,8}: - 
 \frac{7}{288}  : \Omega_{0,4}   \Omega_{3,7}: $$ $$+ \frac{467}{60480}  : \Omega_{1,3}   \Omega_{0,10}: - \frac{13}{756}  : \Omega_{1,3}   \Omega_{1,9}: + 
 \frac{31}{2240}  : \Omega_{1,3}   \Omega_{2,8}: + \frac{47}{1260}  : \Omega_{1,3}   \Omega_{3,7}: $$ $$ -\frac{1}{32}  : \Omega_{1,3}   \Omega_{4,6}: - 
 \frac{1}{525}  : \Omega_{0,5}   \Omega_{1,8}: - \frac{33}{400}  : \Omega_{0,5}   \Omega_{2,7}: + \frac{761}{7200}  : \Omega_{0,5}   \Omega_{3,6}: $$ $$+ 
 \frac{11}{96}  : \Omega_{1,4}   \Omega_{2,7}: - \frac{7}{48}  : \Omega_{1,4}   \Omega_{3,6}: - \frac{27}{448}  : \Omega_{2,3}   \Omega_{0,9}: + 
 \frac{131}{2240}  : \Omega_{2,3}   \Omega_{1,8}:$$ $$ - \frac{31}{84}  : \Omega_{2,3}   \Omega_{2,7}: + \frac{37}{72}  : \Omega_{2,3}   \Omega_{3,6}: - 
 \frac{89}{480}  : \Omega_{2,3}   \Omega_{4,5}: + \frac{11}{720}  : \Omega_{0,6}   \Omega_{1,7}:$$ $$ - \frac{7}{288}  : \Omega_{0,6}   \Omega_{3,5}: - 
 \frac{11}{420}  : \Omega_{1,5}   \Omega_{1,7}: - \frac{3}{160}  : \Omega_{1,5}   \Omega_{2,6}: + \frac{23}{300}  : \Omega_{1,5}   \Omega_{3,5}: $$ $$+ 
 \frac{11}{224}  : \Omega_{2,4}  \Omega_{1,7}: - \frac{7}{80}  : \Omega_{2,4}   \Omega_{3,5}: - \frac{99}{2240}  : \Omega_{0,7}   \Omega_{1,6}: + 
 \frac{11}{192}  : \Omega_{0,7}   \Omega_{2,5}: $$ $$+ \frac{31}{288}  : \Omega_{1,6}   \Omega_{1,6}: - \frac{109}{480}  : \Omega_{1,6}   \Omega_{2,5}: + 
 \frac{35}{576}  : \Omega_{1,6}   \Omega_{3,4}: + \frac{151}{800}  : \Omega_{2,5}   \Omega_{2,5}: $$ $$- \frac{87}{320}  : \Omega_{2,5}   \Omega_{3,4}: + 
 \frac{37}{288}  : \Omega_{3,4}   \Omega_{3,4}: ,$$
 
 $$P^2_0 =  \frac{109}{56000}  \Omega_{0,15} + \partial^2 \bigg( \frac{36613}{26208000}  \Omega_{0,13} + 
  \frac{63901699}{6054048000}  \Omega_{1,12} -  \frac{293340107}{12108096000}  \Omega_{2,11} $$ $$+ 
  \frac{27769129}{1345344000}  \Omega_{3,10}-  \frac{33135533}{403603200}  \Omega_{4,9} + 
 \frac{ 286002151}{1210809600}  \Omega_{5,8}-  \frac{195930023}{605404800}  \Omega_{6,7} \bigg).$$

A calculation using the Mathematica package of Thielemans \cite{T} shows that $\sum_{k=1}^4 P^{2k}_0$ lies in the kernel of $\pi_{-1/2}$, so it must coincide with $P_0$. In particular, $R_0 =  \frac{109}{56000} W^{15}$.

\end{document}